\documentclass{amsart}
\usepackage{amsrefs}
\usepackage[all]{xy}
\allowdisplaybreaks

\newtheorem{theorem}{Theorem}[section]
\newtheorem{lemma}[theorem]{Lemma}
\newtheorem{proposition}[theorem]{Proposition}
\newtheorem{corollary}[theorem]{Corollary}

\theoremstyle{definition}
\newtheorem{definition}[theorem]{Definition}

\theoremstyle{remark}
\newtheorem{remark}[theorem]{Remark}

\numberwithin{equation}{section}

\begin{document}

\title{Stable Ulrich Bundles on Fano Threefolds with Picard Number 2}


\author{Ozhan Genc}
\address{Department of Mathematics\\
Middle East Technical University\\
 Ankara, Turkey}
\curraddr{Mathematics Research and Teaching Group, Middle East Technical University, Northern Cyprus Campus, KKTC, Mersin 10, Turkey}
\email{ozhangenc@gmail.com}
\thanks{The author was supported by project BAP-01-01-2014-001 of the Middle East Technical University and T\"{U}B\.{I}TAK-B\.{I}DEB PhD scholarship (2211)}




\date{\today}

\dedicatory{}

\begin{abstract}

In this paper, we consider the existence problem of rank one and two stable Ulrich bundles on imprimitive Fano 3-folds obtained by blowing-up one of $\mathbb{P}^{3}$, $Q$ (smooth quadric in $\mathbb{P}^{4}$), $V_{3}$ (smooth cubic in $\mathbb{P}^{4}$) or $V_{4}$ (complete intersection of two quadrics in $\mathbb{P}^{5}$) along a smooth irreducible curve. We prove that the only class which admits Ulrich line bundles is the one obtained by blowing up a genus 3, degree 6 curve in $\mathbb{P}^{3}$. Also, we prove that there exist stable rank two Ulrich bundles with $c_{1}=3H$ on a generic member of this deformation class.

\end{abstract}

\maketitle

\section{Introduction}

The existence of Ulrich bundles on smooth projective varieties is related to a number of geometric
questions. For instance, the existence of rank 1 or rank 2 Ulrich bundles on a hypersurface
is related to the representation of that hypersurface as a determinant or Pfaffian (\cite{Bea00}). Another
question of interest is the Minimal Resolution Conjecture (MRC) (\cite{L93}, \cite{FMP03}). In \cite{CKM13}, the existence problem of Ulrich bundles on del Pezzo surfaces was related to the MRC for a general smooth curve in the linear system of the first Chern class of the Ulrich bundle. Also, in \cite{ES11}, it is proved that the cone of cohomology tables of vector bundles on a $k$-dimensonal scheme $X \subset \mathbb{P}^{N}$ is the same as the cone of cohomology tables of vector bundles on $\mathbb{P}^{k}$ if and only if there exists an Ulrich bundle on $X$.

It was conjectured in \cite{DFC03} that on any variety there exist Ulrich bundles. Although it is known that any projective curve (\cite{ES09}), hypersurfaces and complete intersections (\cite{HUB91}), cubic surfaces (\cite{CHGS12}), abelian surfaces (\cite{Bea15}), Veronese varieties (\cite{DFC03}) admit Ulrich bundles, such a general existence result is not known. The finer question of determining the minimal rank of Ulrich bundles (which do not contain bundles of lower rank as direct summands) on a given variety seems to be a quite difficult problem.

The problem that has attracted the most attention is the existence of stable Ulrich
bundles with given rank and Chern classes. Stable Ulrich bundles are particularly interesting as they are the building blocks of all Ulrich bundles: Every Ulrich bundle is semistable, and the Jordan-H\"{o}lder factors are stable Ulrich bundles.

There are very few results on Ulrich bundles over Fano 3-folds with Picard number higher than one like \cite{CFM15}. In this paper, we studied the construction of stable Ulrich bundles on imprimitive Fano 3-folds obtained by blowing-up one of $\mathbb{P}^{3}$, $Q$ (smooth quadric in $\mathbb{P}^{4}$), $V_{3}$ (smooth cubic in $\mathbb{P}^{4}$) or $V_{4}$ (complete intersection of two quadrics in $\mathbb{P}^{5}$) along a smooth irreducible curve. There are 36 deformation classes of Fano 3-folds with Picard number $\rho =2$ and 27 of these are imprimitive (\cite[Table 12.3]{IP99}). Among all imprimitive Fano 3-folds of Picard number $\rho =2$, 21 deformation classes are obtained by blowing-up one of $\mathbb{P}^{3}$, $Q$, $V_{3}$ or $V_{4}$ along a smooth irreducible curve. We focus on these 3-folds and we consider rank 1 and 2 stable Ulrich bundles.

Throughout the paper, we use basically Riemann Roch computations, positivity results, Leray spectral sequence, projection formula, the package RandomSpaceCurve of Macaulay2, Casanellas-Hartshorne extension method and computations of local dimension of Quot Scheme. 

First, we prove that the only class which admits rank 1 Ulrich bundles is the one obtained by blowing up a genus 3, degree 6 curve in $\mathbb{P}^{3}$, which is \cite[No:12 in Table 12.3]{IP99}. These varieties admit two classes of rank 1 Ulrich bundles $\mathcal{L}_{1}$ and $\mathcal{L}_{2}$ (Theorem \ref{line bundles}).

The next step is to construct rank 2 stable Ulrich bundles on these varieties. To do this, we first construct rank 2 simple Ulrich bundles (Theorem \ref{simple rank 2 Ulrich}). For this, we use extensions of rank 1 Ulrich bundles $\mathcal{L}_{1}$ and $\mathcal{L}_{2}$:
\[ 0 \rightarrow \mathcal{L}_{1} \rightarrow \mathcal{E} \rightarrow \mathcal{L}_{2} \rightarrow 0 
\]
or
\[ 0 \rightarrow \mathcal{L}_{2} \rightarrow \mathcal{E} \rightarrow \mathcal{L}_{1} \rightarrow 0. 
\]
Then $\mathcal{E}$ is Ulrich and simple; and it has first Chern class $3H$.

Then, to determine whether there exists a stable Ulrich bundle of rank 2 with $c_{1}=3H$, we use the Quot scheme. It is known that stable vector bundles are simple. We consider the local dimension of the Quot scheme at the simple Ulrich bundle with first Chern class $3H$ and find a lower bound to this dimension (Theorem \ref{dimension of Quot scheme at Ulrich bundle with 3H}). Then we find an upper bound to the dimension of the subset parametrizing the non-stable Ulrich bundles (Proposition \ref{dimension of quot scheme by extensions} and Proposition \ref{dimension of quot scheme by extensions 2}). The latter dimension is strictly smaller than the former; that is, there are stable, rank 2 Ulrich bundles with first Chern class $3H$ (Theorem \ref{stable rank 2 Ulrich bundle}).
\subsection{Notations and Conventions}
We work over an algebraically closed field $\mathbb{K}$ of characteristic 0.
\begin{itemize}
	\item $X$ : Smooth projective variety of degree $c$ and dimension $k$ in ${\mathbb{P}}^{N}$.
  \item $H_{X}$: Hyperplane class of $X$. 
	\item $K_{X}$: Canonical divisor of $X$.
	\item $\mathcal{E}(t)$: The vector bundle $\mathcal{E}\otimes \mathcal{O}_{X}(tH_{X})$ where $\mathcal{E}$ is a vector bundle on $X$, and $t \in \mathbb{Z}$.
	\item $C$: Smooth, irreducible curve of degree $d$ and genus $g$.
	\item $Q$: Smooth quadric in $\mathbb{P}^{4}$
	\item $V_{3}$: Smooth cubic in $\mathbb{P}^{4}$
	\item $V_{4}$: Complete intersection of two quadrics in $\mathbb{P}^{5}$
	\item $\widetilde{X}$: Blow-up of $X$ along $C$.
	\item $\widetilde{Y}$: Non-hyperelliptic Fano 3-fold which is obtained by blowing up one of $\mathbb{P}^{3}$, $Q$, $V_{3}$ or $V_{4}$ along $C$.
	\item $Y$: Deformation class of Fano 3-folds which is obtained by blowing up $\mathbb{P}^{3}$ along a smooth irreducible space curve of degree 6 and genus 3, which is scheme theoretic intersection of cubics.
\end{itemize}
\section*{Acknowledgements}
This article represents my doctoral thesis. I would like to thank my doctoral advisor Emre Coskun, and Hursit Onsiper, both of Middle East Technical University. I would also like to thank Yusuf Mustopa for giving me permission to reproduce his Proposition \ref{square of ideal}.
\section{Preliminaries}

\subsection{Fano Varieties}

\begin{definition}
A smooth projective variety $X$ is called a $Fano$ variety if its anticanonical divisor $-K_{X}$ is ample.
\end{definition}

\begin{definition}
A Fano 3-fold is $imprimitive$ if it is isomorphic to the blow-up of a Fano 3-fold along a smooth irreducible curve.
\end{definition}

The classification of Fano 3-folds with $\rho=2$ has been completed and it can be found in \cite[Table 12.3]{IP99}. In this paper, we consider the question of existence of Ulrich bundles on $\widetilde{Y}$.

Upon blowing-up $X$ along $C$, we have the following commutative diagram:
\[
\xymatrix{
& E \ar[d]_{g} \ar@{^{(}->}[r]^j & \widetilde{X}\ar[d]_{f}\\ 
& C \ar@{^{(}->}[r]^i & X}
\]
the map $f$ is the blow-down map and $E=\mathbb{P}N$ is the exceptional divisor, where $N$ is the normal bundle of $C$ in $X$. Recall that $\widetilde{X}$ stands for $\widetilde{Y}$. Let $h$ be the class of a plane in $A^{1}(X)$, and let $l=h^{2}$ be the class of a line in $A^{2}(X)$. We will denote $\widetilde{h}$ and $\widetilde{l}$ for the pullbacks of $h$ and $l$ to $\widetilde{X}$ respectively; and $e$ denotes the class of the exceptional divisor. Also for any divisor $D\in Z^{1}(C)$, we denote by $F_{D}=g^{*}D\in Z^{1}(C)$ the corresponding linear combination of fibers $E\rightarrow C$, and similarly for divisor classes.
 
\begin{theorem} \label{hilbert polnomial of line bundle}
Let $D=a\widetilde{h}-be$ be a divisor on $\widetilde{Y}=\widetilde{\mathbb{P}}^{3}$, where $a,b \in \mathbb{Z}$. Let $D(t)=D+tH_{\widetilde{Y}}$. Then
\begin{eqnarray*}
\chi(\widetilde{Y},\mathcal{O}(D(t)))&=&\ \ \frac{1}{6}[62-8d+2g]t^{3}\\
&\ &+\frac{1}{6}[(48-3d)a+(6g-12d-6)b-12d+3g+93]t^{2}\\
&\ &+\frac{1}{6}[12a^{2}+(6g-6)b^{2}-6dab+(48-3d)a+(6g-6-12d)b\\
&\ &\ \ \ \ \ +43-4d+g]t\\
&\ &+\frac{1}{6}[a^{3}+(4d+2g-2)b^{3}+6a^{2}+(3g-3da-3)b^{2}+11a\\
&\ &\ \ \ \ \ +(g-3da-4d-1)b+6].
\end{eqnarray*}
\end{theorem}

 \begin{proof}
It is well-known that
	\[K_{\mathbb{P}^{3}}=(-3-1)h=-4h
\]
and
	\[c(T_{\mathbb{P}^{3}})=(1+h)^{1+3}=1+4h+6h^{2}+4h^{3}.
\]
So by \cite[Example 15.4.3]{Ful84}, we have
\begin{eqnarray*}
c_{1}(T_{\widetilde{\mathbb{P}}^{3}})&=&f^{*}c_{1}(T_{\mathbb{P}}^{3})+(1-2)[E]\\
&=&f^{*}(4h)-e\\
&=&4\widetilde{h}-e\\
c_{2}(T_{\widetilde{\mathbb{P}}^{3}})&=&f^{*}c_{2}(T_{\mathbb{P}}^{3})+f^{*}i_{*}[C]-f^{*}c_{1}(T_{\mathbb{P}}^{3})[E]\\
&=&f^{*}(6h^{2})+d\widetilde{l}-f^{*}(4h)e\\
&=&(6+d)\widetilde{h}^{2}-4\widetilde{h}e\\
K_{\widetilde{\mathbb{P}}^{3}}&=&f^{*}K_{\mathbb{P}^{3}}+(2-1)[E]\\
&=&f^{*}(-4h)+e\\
&=&-4\widetilde{h}+e.
\end{eqnarray*}
Then using \cite[Lemma 2.1]{SS84}, we obtain
\begin{eqnarray*}
\widetilde{h}^{3}=1
\end{eqnarray*}
\begin{eqnarray*}
e^{3}=-(-K_{\mathbb{P}^{3}} \cdot C)+2-2g=-(4h\cdot C)+2-2g=-4d-2g+2
\end{eqnarray*}
\begin{eqnarray*}
e^{2} \cdot (-K_{\widetilde{\mathbb{P}}^{3}})=2g-2&\Rightarrow& e^{2}(4\widetilde{h}-e)=2g-2\\
&\Rightarrow&4\widetilde{h}e^{2}-e^{3}=2g-2\\
&\Rightarrow&4\widetilde{h}e^{2}-(-4d-2g+2)=2g-2\\
&\Rightarrow&\widetilde{h}e^{2}=-d
\end{eqnarray*}
\begin{eqnarray*}
e\cdot (-K_{\widetilde{\mathbb{P}}^{3}})^{2}=(-K_{\mathbb{P}^{3}}\cdot C)+2-2g&\Rightarrow&e(4\widetilde{h}-e)^{2}=(4h \cdot C)+2-2g\\
&\Rightarrow&16\widetilde{h}^{2}e-8\widetilde{h}e^{2}+e^{3}=4d+2-2g\\
&\Rightarrow&16\widetilde{h}^{2}e+8d-4d-2g+2=4d+2-2g\\
&\Rightarrow&\widetilde{h}^{2}e=0.
\end{eqnarray*}
Since $\widetilde{Y}=\widetilde{\mathbb{P}}^{3}$ is non-hyperelliptic Fano,
	\[H_{\widetilde{\mathbb{P}}^{3}}=-K_{\widetilde{\mathbb{P}}^{3}}=4\widetilde{h}-e.
\]
Let $D=a\widetilde{h}-be$ be a divisor class on $\widetilde{Y}$. Then
\begin{eqnarray*}
D(t)=D+tH_{\widetilde{Y}}&=&(a\widetilde{h}-be)+t(4\widetilde{h}-e)\\
&=&(a+4t)\widetilde{h}-(b+t)e.
\end{eqnarray*}
Then, we apply Riemann-Roch theorem for line bundles on 3-folds to obtain

\begin{eqnarray*}
\chi(\widetilde{Y},\mathcal{O}(D(t)))&=&\frac{1}{6}(D(t))^{3}+\frac{1}{4}c_{1}(T_{\widetilde{Y}}) \cdot (D(t))^{2}+\frac{1}{12}(c_{1}^{2}(T_{\widetilde{Y}})+c_{2}(T_{\widetilde{Y}})) \cdot (D(t))\\
&&+\frac{1}{24}c_{1}(T_{\widetilde{Y}}) \cdot c_{2}(T_{\widetilde{Y}}).\\
\end{eqnarray*}
Then we have
\begin{eqnarray*}
\chi(\widetilde{Y},\mathcal{O}(D(t)))&=&\ \ \ \frac{1}{6}[(a+4t)\widetilde{h}-(b+t)e]^{3}+\frac{1}{4}[4\widetilde{h}-e][(a+4t)\widetilde{h}-(b+t)e]^{2}\\
&&+\frac{1}{12}[(4\widetilde{h}-e)^{2}+(6+d)\widetilde{h}^{2}-4\widetilde{h}e][(a+4t)\widetilde{h}-(b+t)e]\\
&\ &+\frac{1}{24}[4\widetilde{h}-e][6\widetilde{h}^{2}+d\widetilde{l}-4\widetilde{h}e]\\
&=&\ \ \frac{1}{6}[(a+4t)^{3}\widetilde{h}^{3}-3(a+4t)^{2}(b+t)\widetilde{h}^{2}e+3(a+4t)(b+t)^{2}\widetilde{h}e^{2}\\
&&\ \ \ \ \ -(b+t)^{3}e^{3}]\\
&\ &+\frac{1}{4}[4\widetilde{h}-e][(a+4t)^{2}\widetilde{h}^{2}-2(a+4t)(b+t)\widetilde{h}e+(b+t)^{2}e^{2}]\\
&\ &+\frac{1}{12}[16\widetilde{h}^{2}-8\widetilde{h}e+e^{2}+(6+d)\widetilde{h}^{2}-4\widetilde{h}e][(a+4t)\widetilde{h}-(b+t)e]\\
&\ &+\frac{1}{24}[4\widetilde{h}-e][(6+d)\widetilde{h}^{2}-4\widetilde{h}e].
\end{eqnarray*}
Then, 
\begin{eqnarray*}
\chi(\widetilde{Y},\mathcal{O}(D(t)))&=&\ \ \frac{1}{6}[(a+4t)^{3}-3d(a+4t)(b+t)^{2}-(-4d-2g+2)(b+t)^{3}]\\
&\ &+\frac{1}{4}[4(a+4t)^{2}-4d(b+t)^{2}-2d(a+4t)(b+t)\\
&\ &\ \ \ \ \ \ -(-4d-2g+2)(b+t)^{2}]\\
&\ &+\frac{1}{12}[22(a+4t)-12d(b+t)-(-4d-2g+2)(b+t)]\\
&\ &+\frac{1}{24}[24+4d-4d].\\
\end{eqnarray*}
Now, by expanding, we obtain
\begin{eqnarray*}
\chi(\widetilde{Y},\mathcal{O}(D(t)))&=&\ \ \frac{1}{6}[a^{3}+12a^{2}t+48at^{2}+64t^{3}-3dab^{2}-6dabt-12db^{2}t\\
&\ &\ \ \ \ -3dat^{2}-24dbt^{2}-12dt^{3}+4db^{3}+12db^{2}t+12dbt^{2}+4dt^{3}\\
&\ &\ \ \ \ +2gb^{3}+6gb^{2}t+6gbt^{2}+2gt^{3}-2b^{3}-6b^{2}t-6bt^{2}-2t^{3}]\\
&&+\frac{1}{4}[4a^{2}+32at+64t^{2}-4db^{2}-8dbt-4dt^{2}-2dab-2dat\\
&\ &\ \ \ \ \ -8dbt-8dt^{2}+4db^{2}+8dbt+4dt^{2}+2gb^{2}+4gbt+2gt^{2}\\
&\ &\ \ \ \ \ -2b^{2}-4bt-2t^{2}]\\
&\ &+\frac{1}{12}[22a+88t-12db-12dt+4db+4dt+2gb+2gt\\
&\ &\ \ \ \ \ \ -2b-2t]\\
&\ &+\frac{1}{24}24.
\end{eqnarray*}
Then, collecting the terms with same powers of $t$
\begin{eqnarray*}
\chi(\widetilde{Y},\mathcal{O}(D(t)))&=&\ \ \frac{1}{24}[256-48d+16d+8g-8]t^{3}\\
&\ &+\frac{1}{24}[192a-12da-96db+48db+24gb-24b+384-24d\\
&\ &\ \ \ \ \ \ \ -48d+24d+12g-12]t^{2}\\
&\ &+\frac{1}{24}[48a^{2}-24dab-48db^{2}+48db^{2}+24gb^{2}-24b^{2}+192a\\
&\ &\ \ \ \ \ \ \ -48db-12da-48db+48db+24gb-24b+176-24d\\
&\ &\ \ \ \ \ \ \ +8d+4g-4]t\\
&\ &+\frac{1}{24}[4a^{3}-12dab^{2}+16db^{3}+8gb^{3}-8b^{3}+24a^{2}-24db^{2}\\
&\ &\ \ \ \ \ \ \ -12dab+24db^{2}+12gb^{2}-12b^{2}+44a-24db\\
&\ &\ \ \ \ \ \ \ +8db+4gb-4b+24].
\end{eqnarray*}
Finally,
\begin{eqnarray*}
\chi(\widetilde{Y},\mathcal{O}(D(t)))&=&\ \ \frac{1}{6}[62-8d+2g]t^{3}\\
&\ &+\frac{1}{6}[(48-3d)a+(6g-12d-6)b-12d+3g+93]t^{2}\\
&\ &+\frac{1}{6}[12a^{2}+(6g-6)b^{2}-6dab+(48-3d)a+(6g-6-12d)b\\
&\ &\ \ \ \ \ \ \ +43-4d+g]t\\
&\ &+\frac{1}{6}[a^{3}+(4d+2g-2)b^{3}+6a^{2}+(3g-3da-3)b^{2}+11a\\
&\ &\ \ \ \ \ \ \ +(g-3da-4d-1)b+6].
\end{eqnarray*}

 \end{proof}

\begin{theorem} \label{hilbert polynomial of line bundle on Q}
Let $D=a\widetilde{h}-be$ be a divisor on $\widetilde{Y}=\widetilde{Q}$, where $a,b \in \mathbb{Z}$. Let $D(t)=D+tH_{\widetilde{Q}}$. Then
\begin{eqnarray*}
\chi(\widetilde{Y},\mathcal{O}(D(t)))&=&\ \ \frac{1}{24}[208-24d+8g]t^{3}\\
&\ &+\frac{1}{24}[(216-12d)a+(24g-36d-24)b-36d+12g+312]t^{2}\\
&\ &+\frac{1}{24}[72a^{2}+(24g-24)b^{2}-24dab+(216-12d)a\\
&\ &\ \ \ \ \ \ +(24g-24-36d)b+152-6d+4g]t\\
&\ &+\frac{1}{24}[8a^{3}+(12d+8g-8)b^{3}+36a^{2}+(12g-12da-12)b^{2}\\
&\ &\ \ \ \ \ \ +52a+(4g-12da-12d-4)b+24+3d].
\end{eqnarray*}
\end{theorem}
 \begin{proof}
It follows same pattern in proof of Theorem \ref{hilbert polnomial of line bundle} with minor computational changes.
 \end{proof}

\begin{theorem} \label{hilbert polynomial of line bundle on V3}
Let $D=a\widetilde{h}-be$ be a divisor on $\widetilde{Y}=\widetilde{V}_{3}$, where $a,b \in \mathbb{Z}$. Let $D(t)=D+tH_{\widetilde{V}_{3}}$. Then
\begin{eqnarray*}
\chi(\widetilde{Y},\mathcal{O}(D(t)))&=&\ \ \frac{1}{12}[44-8d+4g]t^{3}\\
&\ &+\frac{1}{12}[(72-6d)a+(12g-12d-12)b-12d+6g+66]t^{2}\\
&\ &+\frac{1}{12}[36a^{2}+(12g-12)b^{2}-12dab+(72-6d)a\\
&\ &\ \ \ \ \ \ +(12g-12-12d)b+46+2g]t\\
&\ &+\frac{1}{12}[6a^{3}+(4d+4g-4)b^{3}+18a^{2}+(6g-6da-6)b^{2}\\
&\ &\ \ \ \ \ \ +(24+2d)a+(2g-6da-4d-2)b+12+2d].
\end{eqnarray*}
\end{theorem}
 \begin{proof}
It follows same pattern in proof of Theorem \ref{hilbert polnomial of line bundle} with minor computational changes.
 \end{proof}

\begin{theorem} \label{hilbert polynomial of line bundle on V4}
Let $D=a\widetilde{h}-be$ be a divisor on $\widetilde{Y}=\widetilde{V}_{4}$, where $a,b \in \mathbb{Z}$. Let $D(t)=D+tH_{\widetilde{V}_{4}}$. Then
\begin{eqnarray*}
\chi(\widetilde{Y},\mathcal{O}(D(t)))&=&\ \ \frac{1}{12}[60-8d+4g]t^{3}\\
&\ &+\frac{1}{12}[(96-6d)a+(12g-12d-12)b-12d+6g+90]t^{2}\\
&\ &+\frac{1}{12}[48a^{2}+(12g-12)b^{2}-12dab+(96-6d)a\\
&\ &\ \ \ \ \ \ +(12g-12-12d)b+54+2g+2d]t\\
&\ &+\frac{1}{12}[8a^{3}+(4d+4g-4)b^{3}+24a^{2}+(6g-6da-6)b^{2}\\
&\ &\ \ \ \ \ \ +(28+3d)a+(2g-6da-4d-2)b+12+3d].
\end{eqnarray*}
\end{theorem}
 \begin{proof}
It follows same pattern in proof of Theorem \ref{hilbert polnomial of line bundle} with minor computational changes.
 \end{proof}

\begin{theorem}[Leray Spectral Sequence]\label{main Leray spectral sequence}
Suppose $\pi : X_{1}\rightarrow X_{2}$ is a morphism of vaieties. Then for any $\mathcal{O}_{X_{1}}$-module $\mathcal{F}$, there is a spectral sequence with $E_{2}$ term given by $H^{p}(X_{2},R^{q}\pi_{*}\mathcal{F})$ abutting to $H^{p+q}(X_{1},\mathcal{F})$.
\end{theorem}

\begin{corollary} \label{Leray spectral sequence}
Let $L$ be a line bundle on $\widetilde{Y}$ and $p+q=k$. Then
\begin{itemize}
	\item $H^{k}(\widetilde{Y},L)=0$ if $H^{p}(Y,R^{q}f_{*}L)=0$ for all possible $p$ and $q$
	\item $H^{k}(\widetilde{Y},L) \cong H^{r}(Y,R^{s}f_{*}L)$ if $H^{p}(Y,R^{q}f_{*}L)=0$ except the tuple $(p,q)=(r,s)$.
\end{itemize}
\end{corollary}
\begin{proof}
It is a direct consequence of Theorem \ref{main Leray spectral sequence}. 
 \end{proof}

\subsection{Ulrich Bundles}

The general references for this section are \cite{CKM13} and \cite{HL10}. 
\begin{definition}
Let $\mathcal{E}$ be a vector bundle on a nonsingular projective variety $X$. Then $\mathcal{E}$ is said to be $semistable$ if for every nonzero subbundle $\mathcal{F}$ of $\mathcal{E}$ we have the inequality 
	\[\frac{P_{\mathcal{F}}}{rank(\mathcal{F})} \leq \frac{P_{\mathcal{E}}}{rank(\mathcal{E})},
\]
where $P_{\mathcal{F}}$ and $P_{\mathcal{E}}$ are the respective Hilbert polynomials and comparison is based on the lexicographic order. It is $stable$ if one always has strict inequality above. 
\end{definition}

\begin{definition}
Let $\mathcal{E}$ be a vector bundle on a nonsingular projective variety $X$. The $slope$ $\mu(\mathcal{E})$ of $\mathcal{E}$ is defined as $deg(\mathcal{E})/rank(\mathcal{E})$. We say that $\mathcal{E}$ is $\mu$-$semistable$ if for every subbundle $\mathcal{F}$ of $\mathcal{E}$ with $0<rank(\mathcal{F})<rank(\mathcal{E})$, we have $\mu(\mathcal{F})\leq \mu(\mathcal{E})$. We say $\mathcal{E}$ is $\mu$-$stable$ if strict inequality always holds above.
\end{definition}

\begin{lemma} \label{stability lemma}
The two definitions are related as follows:
\begin{equation*}
	\mu-stable \Rightarrow stable \Rightarrow semistable \Rightarrow \mu-semistable.
\end{equation*}
\end{lemma}

 \begin{proof}
 See \cite[1.2.13]{HL10}.
 \end{proof}

\begin{definition}
A vector bundle $\mathcal{E}$ on $X$ is called ACM (arithmetically Cohen-Macaulay) if $H^{i}(\mathcal{E}(t))=0$ for all $t \in \mathbb{Z}$ and $0<i<k$.
\end{definition}

\begin{definition}
Let $\mathcal{E}$ be a vector bundle of rank $r$ on $X$. Then $\mathcal{E}$ is \textbf{Ulrich} if for some linear projection $\pi: X\rightarrow \mathbb{P}^{k}$ we have $\pi_{*}\mathcal{E}\cong\mathcal{O}^{cr}_{\mathbb{P}^{k}}$.
\end{definition}

\begin{proposition}\label{hilbert polynoial condition}
Let $\mathcal{E}$ be a vector bundle of rank $r$ on $X$. Then $\mathcal{E}$ is Ulrich if and only if it is ACM with Hilbert polynomial $cr\binom{t+k}{k}$.
\end{proposition}
 \begin{proof}
See \cite[Proposition 2.3]{CKM13}.
 \end{proof}

\begin{theorem}\label{possible Ulrich line bundles on Y}
Let $\widetilde{Y}$ be one of the following Fano 3-folds:
\begin{enumerate}
	\item the blow-up of $\mathbb{P}^{3}$ along an intersection of two cubics,
	\item the blow-up of $\mathbb{P}^{3}$ along a curve of degree 7 and genus 5 which is an intersection of cubics,
	\item the blow-up of $\mathbb{P}^{3}$ along a curve of degree 6 and genus 3 which is an intersection of cubics,
	\item the blow-up of $\mathbb{P}^{3}$ along the intersection of a quadric and a cubic,
	\item the blow-up of $\mathbb{P}^{3}$ along an elliptic curve which is an intersection of two quadrics,
	\item the blow-up of $\mathbb{P}^{3}$ along a twisted cubic,
	\item the blow-up of $\mathbb{P}^{3}$ along a plane cubic,
	\item the blow-up of $\mathbb{P}^{3}$ along a conic,
	\item the blow-up of $\mathbb{P}^{3}$ along a line.
\end{enumerate}
Then Ulrich line bundles can exist only on the class (3).
\end{theorem}

\begin{proof}
Let $D=a\widetilde{h}-be$ be a divisor class on $\widetilde{Y}$. We can compute Hilbert polynomial of $\mathcal{O}_{\widetilde{Y}}(D)$ by Theorem \ref{hilbert polnomial of line bundle}. By Proposition \ref{hilbert polynoial condition}, this must be equal to $deg(\widetilde{Y})\binom{t+3}{3}$. We will equate the coefficients of these two polynomials and try to find integer solutions for $a$ and $b$ in each case separately.
\begin{enumerate}
	
	\item Since $C$ is an intersection of two cubics, $d=9$. By the adjunction formula, $g=10$. Then $k=H^{3}=10$. Now, equate the coefficients of $t^{2}$:
	  \[\frac{10.6}{6}t^{2}=\frac{1}{6}[(48-3.9)a+(6.10-12.9-6)b-12.9+3.10+93]t^{2}
	 \]
	which gives
	  \[a=\frac{18b+15}{7}.
	 \]	
	Next, equate the coefficients of $t$ and use the above relation between $a$ and $b$ to get
		\begin{eqnarray*}
	\frac{10.11}{6}t&=&\frac{1}{6}[12(\frac{18b+15}{7})^{2}+(6.10-6)b^{2}-6.9(\frac{18b+15}{7})b\\
	&&\ \ \ +(48-3.9)(\frac{18b+15}{7})+(6.10-6-12.9)b+43-4.9+10]t\\
	  \end{eqnarray*}
		which gives
		\begin{eqnarray*}
	 b=\frac{3}{2}\mp\frac{7}{30}\sqrt{65}.
	  \end{eqnarray*}
	There is no integer solution for $a$ and $b$, so there exists no Ulrich line bundle.\\
	For the other items except (3), proof follows same pattern in proof of item (1) with minor computational changes. 
	
		\item [(3)](This case is \cite[No.12 in Table 12.3]{IP99}.) It is given that $d=6$ and $g=3$. Then $k=H^{3}=20$. Then equate the coefficients of $t^{2}$:
		\[\frac{20.6}{6}t^{2}=\frac{1}{6}[(48-3.6)a+(6.3-12.6-6)b-12.6+3.3+93]t^{2}
		 \]
		 which gives
		\[ a=2b+3.
	   \]
	Next, equate the coefficients of $t$ and use the above relation between $a$ and $b$ to get
		\begin{eqnarray*}
	\frac{20.11}{6}t&=&\frac{1}{6}[12(2b+3)^{2}+(6.3-6)b^{2}-6.6(2b+3)b+(48-3.6)(2b+3)\\
	&&\ \ \ +(6.3-6-12.6)b+43-4.6+3]t\\
	  \end{eqnarray*}
	  which gives
		\begin{eqnarray*}
	b=0 \ or\  b=3.
	  \end{eqnarray*}
	Then we have $(a,b)=(3,0)$ or $(a,b)=(9,3)$. Both of these solutions satisfy also the equality of coefficients of $t^{2}$ and constant terms. So the divisors $3\widetilde{h}$ and $9\widetilde{h}-3e$ yield possible Ulrich line bundles. (We note that to be Ulrich, they must also satisfy the ACM condition.)	
\end{enumerate}
 
 \end{proof}

\begin{theorem}
Let $\widetilde{Y}$ be one of the following Fano 3-folds:
\begin{enumerate}
 \item the blow-up of $Q$ along the intersection of two divisors from $|\mathcal{O}_{Q}(2)|$,
 \item the blow-up of $Q$ along a curve of degree 6 and genus 2,
 \item the blow-up of $Q$ along an elliptic curve of degree 5,
 \item the blow-up of $Q$ along a twisted quartic,
 \item the blow-up of $Q$ along an intersection of two divisors from $|\mathcal{O}_{Q}(1)|$ and $|\mathcal{O}_{Q}(2)|$,
 \item the blow-up of $Q$ along a conic,
 \item the blow-up of $Q$ along a line.
\end{enumerate}
Then Ulrich line bundles can not exist on non of them.
\end{theorem}
 \begin{proof}
It follows same pattern in proof of Theorem \ref{possible Ulrich line bundles on Y} with minor computational changes.
 \end{proof}

\begin{theorem}
Let $\widetilde{Y}$ be one of the following Fano 3-folds:
\begin{enumerate}
 \item the blow-up of $V_{3}$ along a plane cubic,
 \item the blow-up of $V_{3}$ along a line.
\end{enumerate}
Then Ulrich line bundles can not exist on non of them.
\end{theorem}
 \begin{proof}
It follows same pattern in proof of Theorem \ref{possible Ulrich line bundles on Y} with minor computational changes.	
 \end{proof}

\begin{theorem}
Let $\widetilde{Y}$ be one of the following Fano 3-folds:
\begin{enumerate}
 \item the blow-up of $V_{4}$ along an elliptic curve which is an intersection of two hyperplane sections,
 \item the blow-up of $V_{4}$ along a conic,
 \item the blow-up of $V_{4}$ along a line.
\end{enumerate}
Then Ulrich line bundles can not exist on non of them.
\end{theorem}
 \begin{proof}
It follows same pattern in proof of Theorem \ref{possible Ulrich line bundles on Y} with minor computational changes.
 \end{proof}

\section{Ulrich Line Bundles on $Y$} 
We recall that $Y$ is the Fano 3-fold which is obtained as the blow-up of $\mathbb{P}^{3}$ along a curve $C$ of degree 6 and genus 3.\\
We also recall the following commutative diagram as in 'Preliminaries' section:
\[
\xymatrix{
& E \ar[d]_{g} \ar@{^{(}->}[r]^j & Y\ar[d]_{f}\\ 
& C \ar@{^{(}->}[r]^i & \mathbb{P}^{3}}
\]

\begin{proposition}
The canonical map $\mathcal{O} _{C}\rightarrow g _{\ast }\mathcal{O}_{E}$ is an
isomorphism.
\end{proposition}
 \begin{proof}
Note that $g:E\rightarrow C$ is a ruled surface. Then the result follows from \cite[Lemma 2.1 of Chapter V]{Har77}.
 \end{proof}

\begin{corollary} \label{direct image of exceptional divisor}
 $f_{\ast }(\mathcal{O}_{\widetilde{\mathbb{P}}^{3}}(-m E))=I_{C}^{m}$ and $R^{i}f_{\ast}\mathcal{O}_{\widetilde{\mathbb{P}}^{3}}(-me)=0$ for $m\geq0$ and $i>0$.
\end{corollary}

 \begin{proof}
See \cite[Lemma 4.3.16]{Laz04}.
 \end{proof}

\begin{lemma} \label{canonical bundle of E}
$f_{\ast }\mathcal{O}_{E}(mE)=0$ for $m>0$.
\end{lemma}

 \begin{proof}
Note that $g:E\rightarrow C$ is a ruled surface. So, by \cite[Proposition 8.20 of Chapter II]{Har77}, we have
\begin{eqnarray*}
w_{E}\cong w_{Y} \otimes \mathcal{O}_{Y}(E) \otimes \mathcal{O}_{E}&\Rightarrow&w_{E}\cong\mathcal{O}_{E}(E)(-1)\\
&\Rightarrow&\mathcal{O}_{E}(E)\cong w_{E}(1)\\
&\Rightarrow&\mathcal{O}_{E}(E) \cong \mathcal{O}_{E}(K_{E}+H_{E}).
\end{eqnarray*}
Then we have 
\begin{eqnarray*}
\mathcal{O}_{E}(mE) &\cong& \mathcal{O}_{Y}(mE) \otimes \mathcal{O}_{E}\\
&\cong& \mathcal{O}_{Y}(E)^{\otimes m } \otimes \mathcal{O}_{E}\\
&\cong& [\mathcal{O}_{Y}(E) \otimes \mathcal{O}_{E}]^{\otimes m}\\
&\cong& \mathcal{O}_{E}(E)^{\otimes m }\\
&\cong& \mathcal{O}_{E}(K_{E}+H_{E})^{\otimes m }\\
&\cong& \mathcal{O}_{E}(m(K_{E}+H_{E}))
\end{eqnarray*}
where $D=mK_{E}+mH_{Y}$.\\
Also, by \cite[Lemma 2.10 in Chapter V]{Har10}, we know that  
	\[K_{E} \cong -2C_{0}+D_{C} \cdot F
\]
where $C_{0}$ is a section of the map $g$, $F$ is the fiber of $g$ and $D_{C}$ is a divisor class on $C$.\\
Then
\begin{eqnarray*}
D \cdot F &=& (-2mC_{0}+mD_{C} \cdot F+mH_{E})\cdot F\\
&=& -2mC_{0} \cdot F + mD_{C} \cdot F^{2} + mH_{E} \cdot F\\
&=& -2m+0+m\\
&=& -m.
\end{eqnarray*}
Hence $D \cdot F$ is negative. So, following the proof of \cite[Lemma 2.1 in Chapter V]{Har10}, one can easily show that
	\[f_{\ast }\mathcal{O}_{E}(D)=f_{\ast }\mathcal{O}_{E}(mE)=0.
\]

 \end{proof}

\begin{proposition}
$f _{\ast }(\mathcal{O} _{\widetilde{\mathbb{P}}^{3}}(mE))=\mathcal{O} _{\mathbb{P}^{3}}$ for $m>0.$
\end{proposition}

 \begin{proof}
We have the exact sequence 
 
\begin{center}
$0\rightarrow \mathcal{O} _{\widetilde{\mathbb{P}}^{3}}(-E)\rightarrow \mathcal{O}_{\widetilde{\mathbb{P}}^{3}}\rightarrow \mathcal{O} _{E}\rightarrow 0$.
\end{center}
Now twist this exact sequence by $E$ and get

\begin{center}
$0\rightarrow \mathcal{O} _{\widetilde{\mathbb{P}}^{3}}\rightarrow \mathcal{O} _{\widetilde{\mathbb{P}}^{3}}(E)\rightarrow \mathcal{O} _{E}(E)\rightarrow 0$.~~~~~~~~(*)
\end{center}
Then consider the long exact sequence

\begin{center}
$0\rightarrow f _{\ast }\mathcal{O} _{\widetilde{\mathbb{P}}^{3}}\rightarrow f _{\ast
}\mathcal{O}_{\widetilde{\mathbb{P}}^{3}}(E)\rightarrow f _{\ast }\mathcal{O} _{E}(E)\rightarrow
\cdot \cdot \cdot $.
\end{center}
By Lemma \ref{canonical bundle of E}, $f _{\ast }\mathcal{O} _{E}(E)=0$. So 

\begin{center}
$f _{\ast }\mathcal{O} _{\widetilde{\mathbb{P}}^{3}}(E)\simeq f _{\ast }\mathcal{O} _{\widetilde{\mathbb{P}}^{3}}\simeq \mathcal{O} _{P^{3}}$.
\end{center}
Similarly, now twist the exact sequence (*) by $2E$ and get

\begin{center}
$0\rightarrow \mathcal{O} _{\widetilde{\mathbb{P}}^{3}}(E)\rightarrow \mathcal{O}_{\widetilde{\mathbb{P}}^{3}}(2E)\rightarrow \mathcal{O}_{E}(2E)\rightarrow 0$.
\end{center}
Then consider the exact sequence

\begin{center}
$0\rightarrow f_{\ast }\mathcal{O}_{\widetilde{\mathbb{P}}^{3}}(E)\rightarrow f_{\ast
}\mathcal{O}_{\widetilde{\mathbb{P}}^{3}}(2E)\rightarrow f_{\ast }\mathcal{O}_{E}(2E)\rightarrow
\cdot \cdot \cdot $.
\end{center}
Again by Lemma \ref{canonical bundle of E} $f_{\ast }\mathcal{O}_{E}(2E)=0$. Therefore

\begin{center}
$f_{\ast }\mathcal{O}_{\widetilde{\mathbb{P}}^{3}}(2E)\simeq f_{\ast }\mathcal{O}_{\widetilde{\mathbb{P}}^{3}}(E)\simeq \mathcal{O}_{\mathbb{P}^{3}}$.
\end{center}
Hence, by induction on $m$, we have $f_{\ast }(\mathcal{O}_{\widetilde{\mathbb{P}}^{3}}(mE))=\mathcal{O}_{\mathbb{P}^{3}}$ for $m>0.$
 \end{proof}

\begin{lemma}\label{Ulrich duailty}
Let $\mathcal{E}$ be an Ulrich bundle of rank r on $\widetilde{Y}$. Then $\mathcal{E}^{\vee}(3)$ is also Ulrich.
\end{lemma}

 \begin{proof}
We use Proposition \ref{hilbert polynoial condition}.\\
First, 
\begin{eqnarray*}
(*)\ \ H^{i}(\widetilde{Y},\mathcal{E}^{\vee}(3)(t)) &=& H^{i}(\widetilde{Y},\mathcal{E}^{\vee}(3+t))\\
&=&H^{3-i}(\widetilde{Y},\mathcal{E}(-3-t)\otimes K_{\widetilde{Y}})^{\vee} \ (Serre\ Duality)\\
&=&H^{3-i}(\widetilde{Y},\mathcal{E}(-3-t)\otimes (-H))^{\vee} \ (\widetilde{Y}\ is\ Fano)\\
&=&H^{3-i}(\widetilde{Y},\mathcal{E}(-4-t))^{\vee}.
\end{eqnarray*}
But we know that $\mathcal{E}$ is Ulrich, so it is ACM by Proposition \ref{hilbert polynoial condition}. Then the middle cohomologies of all twists of $\mathcal{E}$ vanish; so $H^{3-i}(\widetilde{Y},\mathcal{E}(-4-t))$ vanishes for $i=1,2$ and $t \in \mathbb{Z}$.\\
Hence $H^{i}(\widetilde{Y},\mathcal{E}^{\vee}(3)(t))=0$; that is, $\mathcal{E}^{\vee}(3)$ is ACM.\\
Second,
\begin{eqnarray*}
\chi(\widetilde{Y},\mathcal{E}^{\vee}(3)(t))&=&\sum^{3}_{i=0}(-1)^{i}h^{i}(\widetilde{Y},\mathcal{E}^{\vee}(3)(t))\\
&=&\sum^{3}_{i=0}(-1)^{i}h^{3-i}(\widetilde{Y},\mathcal{E}(-4-t)) \ \ \ (by\ (*))\\
&=&-cr\binom{-4-t+3}{3}\\
&=&-cr\frac{(-t-1)(-t-2)(-t-3)}{6}\\
&=&cr\frac{(t+1)(t+2)(t+3)}{6}\\
&=&cr\binom{t+3}{3}
\end{eqnarray*}
Therefore $\mathcal{E}^{\vee}(3)$ is Ulrich by Proposition \ref{hilbert polynoial condition}.
 \end{proof}

\begin{lemma}\label{ideal cohomology}
Let $C$ be a curve cut out scheme-theoretically in $\mathbb{P}^{3}$ by cubic hypersurfaces. Then
\begin{center}
$H^{i}( \mathbb{P}^{3},I^{a}_{C|\mathbb{P}^{3}}(k)) =0 $ for $i\geq1$ provided $k\geq3a$.
\end{center}
\end{lemma}

 \begin{proof}
This is a special case of \cite[Proposition 1]{BEL91}.
 \end{proof}

\begin{lemma}\label{resolution}
If $C$ is an ACM curve in $\mathbb{P}^{3}$ with $d=6$ and $g=3$, then its ideal sheaf $I_{C}$ in $\mathbb{P}^{3}$ has the minimal free resolution:
\[0 \rightarrow \mathcal{O}_{\mathbb{P}^{3}}^{\oplus3}(-4) \rightarrow \mathcal{O}_{\mathbb{P}^{3}}^{\oplus4}(-3)\rightarrow I_{C} \rightarrow 0.
\]
\end{lemma}

 \begin{proof}
Since $C$ is ACM, by \cite[p.2]{EI88}, it has a minimal free resolution of the form:
\[0 \rightarrow \bigoplus^{k-1}_{j=1}\mathcal{O}_{\mathbb{P}^{3}}(-n_{j}) \rightarrow \bigoplus^{k}_{l=1}\mathcal{O}_{\mathbb{P}^{3}}(-m_{l})\rightarrow I_{C} \rightarrow 0.
\]
Since $I_{C}(3)$ is generated by global sections \cite[Ex. 8.7(c)]{Har10}, $m_{l}=3$ for all $l$ and we have:
\[(*)\ \ \ \ \ \ \ \ \ \ \ \ \ 0 \rightarrow \bigoplus^{k-1}_{j=1}\mathcal{O}_{\mathbb{P}^{3}}(-n_{j}) \rightarrow \bigoplus^{k}_{l=1}\mathcal{O}_{\mathbb{P}^{3}}(-3)\rightarrow I_{C} \rightarrow 0.
\]
We know that $h^{0}(I_{C}(3))=4$ and $h^{i}(I_{C}(3-i))=0$ for all $i>0$ by \cite[Ex. 8.7(c)]{Har10}.\\ 
Since $h^{i}(I_{C}(3-i))=0$ for all $i>0$, we have $h^{2}(I_{C}(1))=0$.\\ Then $\sum^{k-1}_{j=1}h^{3}(-n_{j}+1)=\sum^{k-1}_{j=1}h^{0}(n_{j}-1-4)=0$. Then $n_{j}\leq4$. But, since (*) is a minimal free resolution, we have $n_{j}\geq4$. So $n_{j}=4$. Since $h^{0}(I_{C}(3))=4$, we have $k=4$. 
 \end{proof}

\begin{proposition}[Yusuf Mustopa, written in private communication]\label{square of ideal}
If $C$ is an ACM space curve with $d=6$ and $g=3$, then $H^{i}(I^{2}_{C}(5))=0$ for all $i>0$.
\end{proposition}

 \begin{proof} 
Twisting the sequence 
\[0 \rightarrow I_{C} \rightarrow \mathcal{O}_{\mathbb{P}^{3}} \rightarrow \mathcal{O}_{C} \rightarrow 0
\]
by $I_{C}(5)$ yields the long exact sequence
\[0 \rightarrow \operatorname{Tor}^{\mathcal{O}_{\mathbb{P}^{3}}}_{1}(I_{C}(5),\mathcal{O}_{C}) \rightarrow I_{C}\otimes I_{C}(5) \rightarrow I_{C}(5) \rightarrow \mathcal{N}^{*}_{C|\mathbb{P}^{3}}(5) \rightarrow 0
\]
This can be broken into two short exact sequences, one of which is 
\[0 \rightarrow \operatorname{Tor}^{\mathcal{O}_{\mathbb{P}^{3}}}_{1}(I_{C}(5),\mathcal{O}_{C}) \rightarrow I_{C}\otimes I_{C}(5) \rightarrow I^{2}_{C}(5) \rightarrow 0.
\]
Since $\operatorname{Tor}^{\mathcal{O}_{\mathbb{P}^{3}}}_{1}(I_{C}(5),\mathcal{O}_{C})$ has at most 1-dimensional support, we have\\
$H^{i}(\operatorname{Tor}^{\mathcal{O}_{\mathbb{P}^{3}}}_{1}(I_{C}(5),\mathcal{O}_{C}))=0$ for all $i>1$. It then suffices to show the vanishing of $H^{i}(I_{C}\otimes I_{C}(5))$ for all $i>0$.\\
We know, by Lemma \ref{resolution}, that  $I_{C}$ has a minimal free resolution of the form 
\[0 \rightarrow \mathcal{O}_{\mathbb{P}^{3}}^{\oplus3}(-4) \rightarrow \mathcal{O}_{\mathbb{P}^{3}}^{\oplus4}(-3)\rightarrow I_{C} \rightarrow 0.
\]
As before, we consider the twist by $I_{C}(5)$. Then we have a long exact sequence
	\[0 \rightarrow \operatorname{Tor}^{\mathcal{O}_{\mathbb{P}^{3}}}_{1}(I_{C}, I_{C}(5)) \rightarrow I_{C}(1)^{\oplus3} \rightarrow I_{C}(2)^{\oplus4} \rightarrow I_{C}\otimes I_{C}(5) \rightarrow 0.
\]
Given that $\operatorname{Tor}^{\mathcal{O}_{\mathbb{P}^{3}}}_{1}(I_{C}, I_{C}(5))$ has at most 1-dimensional support and is a subsheaf of the torsion-free sheaf $I_{C}(1)^{\oplus3}$,  it is equal to 0; so we have
\[0 \rightarrow I_{C}(1)^{\oplus3} \rightarrow I_{C}(2)^{\oplus4} \rightarrow I_{C}\otimes I_{C}(5) \rightarrow 0
\]
But, we know that, by Lemma \ref{resolution}, $I_{C}$ has a minimal free resolution of the form 
\[0 \rightarrow \mathcal{O}_{\mathbb{P}^{3}}^{\oplus3}(-4) \rightarrow \mathcal{O}_{\mathbb{P}^{3}}^{\oplus4}(-3)\rightarrow I_{C} \rightarrow 0.
\]
So $H^{i}(I_{C}(k))=0$ for all $i>0$ and $k>0$. Then $H^{i}(I_{C}\otimes I_{C}(5))=0$ for all $i>0$; so the result follows.
 \end{proof}

\begin{theorem}\label{line bundles}
Suppose that $C$ is ACM. Then there are only two Ulrich line bundles $L_{1}$ and $L_{2}$, and they correspond to divisors $D_{1}=9\widetilde{h}-3e$ and $D_{2}=3\widetilde{h}$ on $Y$.
\end{theorem}

 \begin{proof}
We will use Proposition \ref{hilbert polynoial condition} to show that $L_{1}$ and $L_{2}$ are Ulrich line bundles.
In Theorem \ref{possible Ulrich line bundles on Y}, we showed that $L_{1}$ and $L_{2}$ satisfy the Hilbert polynomial condition. So, it remains to show that $L_{1}$ and $L_{2}$
are ACM; i.e, to show that $H^{1}(Y,L_{1}(t))=H^{2}(Y,L_{1}(t))=0$ and $H^{1}(Y,L_{2}(t))=H^{2}(Y,L_{2}(t))=0$ for all $t\in \mathbb{Z}$.

Consider $L_{1}$ first.

\begin{itemize}
	\item$t\geq0$:
\end{itemize}
Then 
\begin{center}
$L_{1}(t)=\mathcal{O}_{Y}(9\widetilde{h}-3e+t(4\widetilde{h}-e))=\mathcal{O}_{Y}((4t+9)\widetilde{h}+(-t-3)e)$.
\end{center}
Since $-t-3 < 0$, by the projection formula and Corollary \ref{direct image of exceptional divisor}, we have
\begin{eqnarray*}
f_{\ast }L_{1}( t)&=&f_{\ast }\mathcal{O}_{Y}( ( 4t+9) \widetilde{h}%
+( -t-3) e)  \\
&=&\mathcal{O}_{\mathbb{P}^{3}}( 4t+9) \otimes f_{\ast }\mathcal{O}_{Y}( ( -t-3) e)\\
&=&I_{C}^{t+3}\otimes \mathcal{O}_{\mathbb{P}^{3}}( 4t+9)\\
&=&I_{C}^{t+3}( 4t+9).
\end{eqnarray*}
So $H^{1}(\mathbb{P}^{3},f_{\ast }L_{1}( t))=H^{1}(\mathbb{P}^{3},I_{C}^{t+3}( 4t+9))$ and it is $0$ by Lemma \ref{ideal cohomology}, since\\
 $4t+9 \geq 3(t+3)$.\\
Now we consider 
$H^{0}(\mathbb{P}^{3},R^{1}f_{\ast }L_{1}( t) ).$ By projection formula, we have
\begin{center}
 $H^{0}( \mathbb{P}^{3},R^{1}f_{\ast }L_{1}( t) )
 =H^{0}( \mathbb{P}^{3},R^{1}f_{\ast }\mathcal{O}_{\mathbb{P}^{3}}( ( -t-3)
 e) \otimes \mathcal{O}_{\mathbb{P}^{3}}( 4t+9)) $
\end{center}
But $R^{1}f_{\ast }\mathcal{O}_{\mathbb{P}^{3}}(( -t-3) e)=0$ by Corollary \ref{direct image of exceptional divisor}, since $-t-3\leq 0$; and so\\
$H^{0}(\mathbb{P}^{3},R^{1}f_{\ast }L_{1}( t) )=0$. Then, by Corollary \ref{Leray spectral sequence}, we have
\begin{center}
 $H^{1}(Y,L_{1}(t) ) =0.$
\end{center}
Now consider $H^{0}( \mathbb{P}^{3},R^{2}f_{\ast }L_{1}( t))$ and $H^{1}( \mathbb{P}^{3},R^{1}f_{\ast }L_{1}(t))$. Note that they are $0$ by Corollary \ref{direct image of exceptional divisor}. Also $H^{2}(\mathbb{P}^{3},f_{\ast }L_{1}( t))=H^{2}(\mathbb{P}^{3},I_{C}^{t}(4t+3))=0$ again by Lemma \ref{ideal cohomology}. Then, by Corollary \ref{Leray spectral sequence}, we have
\begin{center}
 $H^{2}( Y,L_{1}(t) ) =0$.
\end{center}
\begin{itemize}
	\item$t<-4$:
\end{itemize}
Then
\begin{eqnarray*}
H^{1}(Y,L_{1}(t)) &=& H^{2}(Y,L_{1}^{\vee}(-t)\otimes K_{Y})^{\vee}\\
&=& H^{2}(Y,\mathcal{O}_{Y} ((-4t-9)\widetilde{h}+(t+3)e)\otimes \mathcal{O}_{Y}(-4\widetilde{h}+e))^{\vee}\\
&=& H^{2}(Y,\mathcal{O}_{Y}((-4t-13)\widetilde{h}+(t+4)e))^{\vee}.\\
\end{eqnarray*}
Similarly,
\begin{eqnarray*}
H^{2}(Y,L_{1}(t)) &=& H^{1}(Y,\mathcal{O}_{Y}((-4t-13)\widetilde{h}+(t+4)e))^{\vee}.
\end{eqnarray*}
So, if $H^{i}(Y,\mathcal{O}_{Y}((-4t-13)\widetilde{h}+(t+4)e))$ for $i=1,2$ vanishes, the result will follow.\\
Since $t+4 < 0$, by the projection formula and Corollary \ref{direct image of exceptional divisor}, we have
\begin{eqnarray*}
f_{\ast }\mathcal{O}_{Y}((-4t-13)\widetilde{h}+(t+4)e))&=&\mathcal{O}_{\mathbb{P}^{3}}(-4t-13) \otimes f_{\ast }\mathcal{O}_{Y}((t+4)e)\\
&=&I_{C}^{-t-4}\otimes \mathcal{O}_{\mathbb{P}^{3}}(-4t-13)\\
&=&I_{C}^{-t-4}(-4t-13).
\end{eqnarray*}
So $H^{1}(\mathbb{P}^{3},f_{\ast }\mathcal{O}_{Y}((-4t-13)\widetilde{h}+(t+4)e))=H^{1}(\mathbb{P}^{3},I_{C}^{-t-4}(-4t-13))$ and it is $0$ by Lemma \ref{ideal cohomology}, since $-4t-13\geq 3(-t-4)$.\\
Now consider 
$H^{0}(\mathbb{P}^{3},R^{1}f_{\ast }\mathcal{O}_{Y}((-4t-13)\widetilde{h}+(t+4)e) ).$ By the projection formula, we have
\begin{center}
 $H^{0}( \mathbb{P}^{3},R^{1}f_{\ast }\mathcal{O}_{Y}((-4t-13)\widetilde{h}+(t+4)e))
 =H^{0}( \mathbb{P}^{3},R^{1}f_{\ast }\mathcal{O}_{\mathbb{P}^{3}}( (t+4)
 e) \otimes \mathcal{O}_{\mathbb{P}^{3}}(-4t-13)). $
\end{center}
But $R^{1}f_{\ast }\mathcal{O}_{\mathbb{P}^{3}}((t+4)e)=0$ by Corollary \ref{direct image of exceptional divisor}, since $t+4\leq0$. So, by Corollary \ref{Leray spectral sequence}, $H^{1}(Y,\mathcal{O}_{Y}((-4t-13)\widetilde{h}+(t+4)e)) =0.$
So,
	\[H^{2}(Y,L_{1}(t))=0.
\]
Now consider $H^{0}( \mathbb{P}^{3},R^{2}f_{\ast }\mathcal{O}_{Y}((-4t-13)\widetilde{h}+(t+4)e))$ and\\
 $H^{1}( \mathbb{P}^{3},R^{1}f_{\ast }\mathcal{O}_{Y}((-4t-13)\widetilde{h}+(t+4)e))$, and note that they are $0$ by Corollary \ref{direct image of exceptional divisor}.\\
Also $H^{2}(\mathbb{P}^{3},f_{\ast }\mathcal{O}_{Y}((-4t-13)\widetilde{h}+(t+4)e))=H^{2}(\mathbb{P}^{3},I_{C}^{-t-4}(-4t-13))=0$ again by Lemma \ref{ideal cohomology}. So, by Corollary \ref{Leray spectral sequence}, $H^{2}(Y,\mathcal{O}_{Y}((-4t-13)\widetilde{h}+(t+4)e))$ vanish. So
\begin{center}
 $H^{1}( Y,L_{1}(t) ) =0$.
\end{center}

\begin{itemize}
	\item $t=-4$:
\end{itemize}
Then
	\[H^{i}(Y,L_{1}(-4))=H^{i}(Y,\mathcal{O}_{Y}(-7\widetilde{h}+e)).
\]
Then by \cite[Lemma 1.4]{BEL91}, we have
\begin{eqnarray*}
H^{i}(Y,\mathcal{O}_{Y}(-7\widetilde{h}+e)) &=& H^{i}(\mathbb{P}^{3},\mathcal{O}_{\mathbb{P}^{3}}(7h)).
\end{eqnarray*}
So,
\begin{center}
	 $H^{1}(Y,L_{1}(-4))=H^{2}(Y,L_{1}(-4))=0$.
\end{center}
\begin{itemize}	
	\item $t=-3$:
\end{itemize}
Then
	\[H^{i}(Y,L_{1}(-3))=H^{i}(Y,\mathcal{O}_{Y}(-3\widetilde{h})).
\]
Then by \cite[Lemma 1.4]{BEL91}, we have
\begin{eqnarray*}
H^{i}(Y,\mathcal{O}_{Y}(-3\widetilde{h})) &=& H^{i}(\mathbb{P}^{3},\mathcal{O}_{\mathbb{P}^{3}}(3h)).
\end{eqnarray*}
So,
\begin{center}
	 $H^{1}(Y,L_{1}(-3))=H^{2}(Y,L_{1}(-3))=0$.
\end{center}
So far, we showed that $H^{1}(Y,L_{1}(t))=H^{2}(Y,L_{1}(t))=0$ for all $t$ except $t=-1,-2$.
For the remaining two values of $t$, we assume that $C$ is ACM.
\begin{itemize}
	\item $t=-1$:
\end{itemize}
Again by Corollary \ref{Leray spectral sequence}, if $H^{i}(\mathbb{P}^{3},f_{\ast }L_{1}(-1))$ for $i=1,2$,\\
$H^{j}(\mathbb{P}^{3},R^{1}f_{\ast }L_{1}(-1))$ for $j=0,1$ and $H^{0}(\mathbb{P}^{3},R^{2}f_{\ast }L_{1}(-1))$ vanishes,\\
 then $H^{1}( Y,L_{1}(-1) )$ and $H^{2}( Y,L_{1}(-1) )$ vanish.\\
Note that $H^{i}(\mathbb{P}^{3},f_{\ast }L_{1}( -1))=H^{i}(\mathbb{P}^{3},I_{C}^{2}( 5))$ for $i=1,2$ by the projection formula and Corollary \ref{direct image of exceptional divisor}; and we know that $H^{i}(\mathbb{P}^{3},I_{C}^{2}( 5))=0$ for $i=1,2$ by Proposition \ref{square of ideal}. Also, we know that $H^{j}(\mathbb{P}^{3},R^{1}f_{\ast }L_{1}(-1))=0$ for $j=0,1$ and\\
 $H^{0}(\mathbb{P}^{3},R^{2}f_{\ast }L_{1}(-1))=0$ by the projection formula and Corollary \ref{direct image of exceptional divisor}. So,
\begin{center}
	 $H^{1}(Y,L_{1}(-1))=H^{2}(Y,L_{1}(-1))=0$.
\end{center}
\begin{itemize}
	\item $t=-2$:
\end{itemize}
By Corollary \ref{Leray spectral sequence}, if all of $H^{i}(\mathbb{P}^{3},f_{\ast }L_{1}(-2))$ for $i=1,2$,\\
$H^{j}(\mathbb{P}^{3},R^{1}f_{\ast }L_{1}(-2))$ for $j=0,1$ and $H^{0}(\mathbb{P}^{3},R^{2}f_{\ast }L_{1}(-2))$ vanish,\\
then $H^{1}( Y,L_{1}(-2) )$ and $H^{2}( Y,L_{1}(-2) )$ vanish.\\
Note that $H^{i}(\mathbb{P}^{3},f_{\ast }L_{1}(-2))=H^{i}(\mathbb{P}^{3},I_{C}(1))$ for $i=1,2$ by the projection formula and Corollary \ref{direct image of exceptional divisor}. But $H^{i}(\mathbb{P}^{3},I_{C}(1))=0$ for $i=1,2$ since $C$ is ACM and by Lemma \ref{resolution} $I_{C}$ has a minimal free resolution
\[0 \rightarrow \mathcal{O}_{\mathbb{P}^{3}}^{\oplus3}(-4) \rightarrow \mathcal{O}_{\mathbb{P}^{3}}^{\oplus4}(-3)\rightarrow I_{C} \rightarrow 0.
\]
Also, we know that $H^{j}(\mathbb{P}^{3},R^{1}f_{\ast }L_{1}(-2))=0$ for $j=0,1$ and\\
 $H^{0}(\mathbb{P}^{3},R^{2}f_{\ast }L_{1}(-2))=0$ by projection formula and Corollary \ref{direct image of exceptional divisor}. So,
\begin{center}
	 $H^{1}(Y,L_{1}(-2))=H^{2}(Y,L_{1}(-2))=0$.
\end{center}
Hence $H^{1}(Y,L_{1}(t))=H^{2}(Y,L_{1}(t))=0$ for all $t\in \mathbb{Z}$, and the result follows for $L_{1}$.

Consider $L_{2}$ next.\\
$L_{2}$ is Ulrich by Lemma \ref{Ulrich duailty}, since 
\begin{center}
	$L_{1}^{\vee}(3)=(-(9)\widetilde{h}+3e)+3(4\widetilde{h}-e)=3\widetilde{h}=L_{2}.$
\end{center}
 \end{proof}

\begin{remark} \label{irreducibility of hilbert subscheme}
We know that $H_{6,3,3}$, which is the open subscheme of the Hilbert Scheme parametrizing the smooth irreducible curves of $d=6$ and $g=3$ in $\mathbb{P}^{3}$, is irreducible by \cite[Theorem 4]{E86}. Also, we know that the property of being an ACM sheaf is an open condition by \cite{CHGS12}. Hence, if we assume $C$ is ACM, then the line bundles $L_{1}$ and $L_{2}$ exist on a generic element of the deformation class $Y$.
\end{remark}

\section{Rank 2 Ulrich Bundles on $Y$}
Let $E$ be a vector bundle of rank r, and $L$ a line bundle on $X$. Then, by \cite[Ex. 3.2.2]{Ful84}, for all $p\geq0$,
  \[c_{p}(E \otimes L)=\sum^{p}_{i=0} {{r-i}\choose{p-i}}c_{i}(E) \cdot c_{1}^{p-i}(L).
\]
Then, if $E$ is a rank 2 vector bundle, we have
\begin{eqnarray*}
c_{1}(E\otimes\mathcal{O}_{X}(tH))&=&\sum^{1}_{i=0} {{2-i}\choose{1-i}}c_{i}(E) \cdot c_{1}^{1-i}(\mathcal{O}_{X}(tH))\\
&=&2c_{0}(E) \cdot c_{1}(\mathcal{O}_{X}(tH))+c_{1}(E)\\
&=&c_{1}(E)+2tH
\end{eqnarray*}
and
\begin{eqnarray*}
c_{2}(E\otimes \mathcal{O}_{X}(tH))&=&\sum^{2}_{i=0} {{2-i}\choose{2-i}}c_{i}(E) \cdot c_{1}^{2-i}(\mathcal{O}_{X}(tH))\\
&=&\sum^{2}_{i=0} c_{i}(E) \cdot c_{1}^{2-i}(\mathcal{O}_{X}(tH))\\
&=&c_{0}(E) \cdot c_{1}^{2}(\mathcal{O}_{X}(tH))+c_{1}(E) \cdot c_{1}(\mathcal{O}_{X}(tH))+c_{2}(E)\\
&=&(tH)^{2}+tc_{1}(E) \cdot H+c_{2}(E).
\end{eqnarray*}

\begin{theorem}\label{rank 2 Ulrich bundle's chern relations}
Let $(\widetilde{Y},H)$ be a Fano threefold which is the blow-up of $\mathbb{P}^{3}$ along a smooth, irreducible curve of degree $d$ and genus $g$. If $\mathcal{E}$ is a rank 2 Ulrich bundle on $\widetilde{Y}$, then we have
\begin{enumerate}
	\item $H^{2} \cdot c_{1}(\mathcal{E})=3H^{3}$,
	\item $H \cdot c_{2}(\mathcal{E})=\frac{1}{2}H \cdot c_{1}^{2}(\mathcal{E})-2H^{3}+4$,
	\item $2c^{3}_{1}(\mathcal{E})-6c_{1}(\mathcal{E}) \cdot c_{2}(\mathcal{E})+c_{1}(\mathcal{E}) \cdot c_{2}(K_{\widetilde{Y}})=9H^{3}$.
\end{enumerate}
\end{theorem}

 \begin{proof}
Let $c_{i}=c_{i}(\mathcal{E})$ and $d_{i}=c_{i}(\mathcal{K}_{\widetilde{Y}})$. Then, by Riemann-Roch theorem
\begin{eqnarray*}
\chi(\widetilde{Y},\mathcal{E}(t))&=&\ \ \frac{1}{6}[(2tH+c_{1})^{3}-3(2tH+c_{1})((tH)^{2}+tc_{1} \cdot H+c_{2})]\\
&\ &+\frac{1}{4}H[(2tH+c_{1})^{2}-2((tH)^{2}+tc_{1} \cdot H+c_{2})]\\
&\ &+\frac{1}{12}(H^{2}+d_{2})(2tH+c_{1})+\frac{1}{12}Hd_{2}\\
&=&\ \ \frac{1}{6}(8H^{3}t^{3}+12H^{2} \cdot c_{1}t^{2}+6H \cdot c^{2}_{1}t+c^{3}_{1}-6H^{3}t^{3}-6H^{2} \cdot c_{1}t^{2}\\
&\ &\ \ \ \ -6H \cdot c_{2}t-3H^{2} \cdot c_{1}t^{2}-3H \cdot c^{2}_{1}t-3c_{1} \cdot c_{2})\\
&\ &+\frac{1}{4}(4H^{3}t^{2}+4H^{2} \cdot c_{1}t+H \cdot c^{2}_{1}-2H^{3}t^{2}-2H^{2} \cdot c_{1}t-2H \cdot c_{2})\\
&\ &+\frac{1}{12}(2H^{3}t+H^{2} \cdot c_{1}+2H \cdot d_{2}t+c_{1} \cdot d_{2})+\frac{1}{12}(H \cdot d_{2})\\
&=&\ \ \frac{1}{3}H^{3}t^{3}+(\frac{1}{2}H^{2} \cdot c_{1}+\frac{1}{2}H^{3})t^{2}\\
&\ &+(\frac{1}{2}H \cdot c_{1}^{2}-H \cdot c_{2}+\frac{1}{2}H^{2} \cdot c_{1}+\frac{1}{6}H^{3}+\frac{1}{6}H \cdot d_{2})t\\
&\ & +(\frac{1}{6} c_{1}^{3}-\frac{1}{2}c_{1} \cdot c_{2}+\frac{1}{4}H \cdot c_{1}^{2}-\frac{1}{2}H \cdot c_{2}+\frac{1}{12} H^{2} \cdot c_{1}\\
&\ &\ \ \ \ +\frac{1}{12} c_{1} \cdot d_{2}+\frac{1}{12} H \cdot d_{2}).
\end{eqnarray*}
Since $\mathcal{E}$ is a rank 2 Ulrich bundle, by Proposition \ref{hilbert polynoial condition}, we have
	\[\chi(\widetilde{Y},\mathcal{E}(t))=2H^{3}{{t+3}\choose{3}}=H^{3}\frac{(t^{3}+6t^{2}+11t+6)}{3}.
\]
So, if we equate coefficients of $t^{2}$, we get 
	\begin {eqnarray*} 
\frac{1}{2}H^{2} \cdot c_{1}+\frac{1}{2}H^{3}=2H^{3}
\end {eqnarray*}
\begin{eqnarray*}
\Rightarrow H^{2} \cdot c_{1}=3H^{3}.
\end {eqnarray*}
If we equate coefficients of $t$, we get
 	\begin {eqnarray*} 
\frac{1}{2}H \cdot c_{1}^{2}-H \cdot c_{2}+\frac{1}{2}H^{2} \cdot c_{1}+\frac{1}{6}H^{3}+\frac{1}{6}H \cdot d_{2}=\frac{11}{3}H^{3}
\end {eqnarray*}
\begin{eqnarray*}
&\Rightarrow&H \cdot c_{2}=\frac{1}{2}H \cdot c_{1}^{2}-2H^{3}+\frac{1}{6}H \cdot d_{2}\ \ \ \ \ \ \ \ \ \ \ \ \ \ \ \ \ \ \ \ \ \ \ \ \ \ \ \ \ \ \ \ \ \ \ (by\ part\ (1))\\
&\Rightarrow&H \cdot c_{2}=\frac{1}{2}H \cdot c_{1}^{2}-2H^{3}+4.
\end {eqnarray*}
If we equate constant terms, we get
	\begin {eqnarray*} 
\frac{1}{6}c_{1}^{3}-\frac{1}{2}c_{1} \cdot c_{2}+\frac{1}{4}H \cdot c_{1}^{2}-\frac{1}{2}H \cdot c_{2}+\frac{1}{12}H^{2} \cdot c_{1}+\frac{1}{12}c_{1} \cdot d_{2}+\frac{1}{12}H \cdot d_{2}=2H^{3}
\end {eqnarray*}
\begin{eqnarray*}
&\Rightarrow&2c^{3}_{1}-6c_{1} \cdot c_{2}+c_{1} \cdot d_{2}=9H^{3}.
\end {eqnarray*}
 \end{proof}

\begin{theorem}
Let $\mathcal{E}$ be a rank 2 Ulrich bundle on $Y$ with $c_{1}(Y)=x\widetilde{h}-ye$. Then there are 7 possibilities for $c_{1}(Y)$, which are
\begin{itemize}
	\item 6$\widetilde{h}$,
	\item 8$\widetilde{h}$-e,
	\item 10$\widetilde{h}$-2e,
	\item 12$\widetilde{h}$-3e=3$H$,
	\item 14$\widetilde{h}$-4e,
	\item 16$\widetilde{h}$-5e,
	\item 18$\widetilde{h}$-6e.
\end{itemize}
 
\end{theorem}
 \begin{proof}
We know that $H_{Y}=4\widetilde{h}-e$. By Theorem \ref{rank 2 Ulrich bundle's chern relations}
\begin{eqnarray*}
&&(4\widetilde{h}-e)^{2}(x\widetilde{h}-ye)=3(4\widetilde{h}-e)^{3}\\
&\Rightarrow&16x+8y(-6)+x(-6)-y(-28)=3.20\ \ \ \ \ \ \ \ \ \ \ \ \ \ \ \ \ \ \ \ \ \ \ \ \ (Theorem\ \ref{hilbert polnomial of line bundle})\\
&\Rightarrow&x=2y+6.
\end{eqnarray*}
Since $\mathcal{E}$ is Ulrich, it is $\mu$-semistable by Theorem \ref{stability for Ulrich}. So, we can apply Bogomolov's Inequality \cite[Theorem 7.3.1]{HL10} and get 
\begin{eqnarray*}
&&(2.2c_{2}(\mathcal{E})-(2-1)c_{1}^{2}(\mathcal{E}))H\geq 0\\
&\Rightarrow&4Hc_{2}(\mathcal{E})-Hc_{1}^{2}(\mathcal{E})\geq 0\\
&\Rightarrow&4(H\frac{c_{1}^{2}(\mathcal{E})}{2}-2H^{3}+4)-Hc_{1}^{2}(\mathcal{E})\geq 0\ \ \ \ \ \ \ \ \ \ \ \ \ \ \ \ \ \ \ \ \ \ \ \ \ \ \ \ \ (Theorem \ \ref{rank 2 Ulrich bundle's chern relations})\\
&\Rightarrow&Hc_{1}^{2}(\mathcal{E})-8H^{3}+16\geq 0\\
&\Rightarrow&(4\widetilde{h}-e)(x\widetilde{h}-ye)^{2}-8.20+16\geq 0\\
&\Rightarrow&4x^{2}+4y^{2}(-6)+2xy(-6)-y^{2}(-28)-144\geq 0\ \ \ \ \ \ \ \ \ \ \ \ \ \ \ (Theorem\ \ref{hilbert polnomial of line bundle})\\
&\Rightarrow&4(2y+6)^{2}-24y^{2}-12y(2y+6)+28y^{2}-144\geq 0\\
&\Rightarrow&-4y^{2}+24y\geq0\\
&\Rightarrow&0 \leq y \leq 6.
\end{eqnarray*}
 \end{proof}

\subsection{Simple Ulrich Bundles on $Y$ with $c_{1}=3H$}

\begin{proposition}\label{upper semicont}
Let $X$ be projective variety of dimension $k$ in $\mathbb{P}^{N}$ and $I_{X}$ be the ideal sheaf of $X$ in $\mathbb{P}^{N}$. Then $H^{i}(\mathbb{P}^{N},I^{n}_{X}(t))$ is upper semi-continuous for $i>k$.
\end{proposition}
 \begin{proof}
Twisting the sequence 
\[0 \rightarrow I_{X} \rightarrow \mathcal{O}_{\mathbb{P}^{N}} \rightarrow \mathcal{O}_{X} \rightarrow 0
\]
by $I_{X}^{\otimes n-1}(t)$ yields the long exact sequence
\[0 \rightarrow \operatorname{Tor}^{\mathcal{O}_{\mathbb{P}^{N}}}_{1}(I_{X}^{\otimes n-1}(t),\mathcal{O}_{X}) \rightarrow I_{X}^{\otimes n}(t) \rightarrow I_{X}^{\otimes n-1}(t) \rightarrow \mathcal{O}_{X} \otimes I_{X}^{\otimes n-1}(t) \rightarrow 0
\]
This can be broken into two short exact sequences, one of which is 
\[0 \rightarrow \operatorname{Tor}^{\mathcal{O}_{\mathbb{P}^{N}}}_{1}(I_{X}^{\otimes n-1}(t),\mathcal{O}_{X}) \rightarrow I_{X}^{\otimes n}(t) \rightarrow I^{n}_{X}(t) \rightarrow 0.
\]
Since $\operatorname{Tor}^{\mathcal{O}_{\mathbb{P}^{N}}}_{1}(I_{X}^{\otimes n-1}(t),\mathcal{O}_{X})$ has at most k-dimensional support, we have\\
$H^{i}(\mathbb{P}^{N},\operatorname{Tor}^{\mathcal{O}_{\mathbb{P}^{N}}}_{1}(I_{X}^{\otimes n-1}(t),\mathcal{O}_{X}))=0$ for all $i>k$. So, by long exact sequence of cohomology, we get $H^{i}(\mathbb{P}^{N},I_{X}^{\otimes n}(t))=H^{i}(\mathbb{P}^{N},I^{n}_{X}(t))$ for all $i>k$. Since left hand side is upper semi-continuous, the right hand side is upper semi-continuous.
 \end{proof}
\begin{theorem}\label{macaulay}
Let $C$ be an smooth ACM space curve with $d=6$ and $g=3$. Then $h^{2}(\mathbb{P}^{3},I^{3}_{C}(6))=0$ and $h^{2}(\mathbb{P}^{3},I^{2}_{C}(2))\leq 8$ for a generic such $C$.
\end{theorem}

 \begin{proof}
Use \textit{Macaulay2} \cite{mac} for computations: 
{\small
\begin{verbatim}
     i1 : k = ZZ/32467; R = k[x,y,z,w];
\end{verbatim}}
Then load the package \textit{RandomSpaceCurves} \cite{macp} to produce explicit example of smooth ACM space curve $C^{'}$ of $d=6$ and $g=3$ with ideal $J$:
{\small
\begin{verbatim}
     i2 : load"RandomSpaceCurves.m2";
     i3 : J=(random spaceCurve)(6,3,R)
     o3 : ideal (-2215x^3+10620x^2y+2508xy^2-15048y^3-5453x^2z-2767xyz
             +8885y^2z+2225xz^2+1759yz^2-9499z^3+3014x^2w+12412xyw
             -1419y^2w-11910xzw-3506yzw-831z^2w-1546xw^2+4414yw^2
             -10576zw^2+15249w^3, -6292x^3-10864x^2y+5626xy^2-8024y^3
             +10837x^2z-6966xyz+9956y^2z-9501xz^2-9538yz^2+9745z^3
             +15655x^2w-3220xyw-12116y^2w+11148xzw-3392yzw-1539z^2w
             -3915xw^2-5992yw^2+15589zw^2+7309w^3, 870x^3+9582x^2y
             -172xy^2+8082y^3-13952x^2z+1923xyz+13352y^2z+7141xz^2
             -13354yz^2+15747z^3+1042x^2w+1494xyw-11584y^2w+7730xzw
             -4628yzw+9837z^2w-4220xw^2+4893yw^2-15379zw^2-13719w^3, 
             -15941x^3-8361x^2y-16223xy^2+12866y^3-4501x^2z+13591xyz
             -11196y^2z-6043xz^2-7842yz^2+11284z^3+1057x^2w-2552xyw
             +6508y^2w+15994xzw-2374yzw-10280z^2w+7766xw^2+15317yw^2
             -10555zw^2+7241w^3)
     o3 : Ideal of R 
\end{verbatim}}
Then check whether $C^{'}$ is a smooth ACM space curve of $d=6$ and $g=3$:
{\small
\begin{verbatim}
     i4 : (degree J, genus J, resolution J)
     o4 : (6, 3, R^1 <-- R^4 <-- R^3 <-- O)
                 0       1       2       3
     o4 : Sequence
\end{verbatim}}
Then compute $h^{2}(J^{3}_{C^{'}}(6))$ and $h^{2}(J^{2}_{C^{'}}(2))$:
{\small
\begin{verbatim}
     i5 : J3 = J*J*J;
          J2 = J*J; 
          vJ3 = Proj(R/J3); 
          vJ2 = Proj(R/J2); 
          sJ3 = sheaf module ideal vJ3;
          sJ2 = sheaf module ideal vJ2;
     o5 : Ideal of R
     o6 : Ideal of R
     i11: (HH^2 (sJ3(6)), HH^2 (sJ2(2)))
     o11: (0, k^8)
     o11: Sequence
\end{verbatim}}
But, we know that these cohomologies are upper semi-continuous functions by Proposition \ref{upper semicont}. Hence, we have $h^{2}(\mathbb{P}^{3},I^{3}_{C}(6))=0$ and $h^{2}(\mathbb{P}^{3},I^{2}_{C}(2))\leq 8$ for a generic element of all smooth ACM space curves of $d=6$ and $g=3$.	

 \end{proof}

\begin{remark} \label{open condition}
Since cohomology is an upper semi-continuous function, as stated in the proof of Theorem \ref{macaulay}, smooth ACM space curves of $d=6$ and $g=3$ satisfying $h^{2}(I^{3}_{C}(6))=0$ form an open subset of all smooth ACM space curves of $d=6$ and $g=3$. Also by Remark \ref{irreducibility of hilbert subscheme}, we know that $H_{6,3,3}$ is irreducible and smooth ACM space curves of $d=6$ and $g=3$ form an open subset in $H_{6,3,3}$. So, smooth ACM space curves of $d=6$ and $g=3$ satisfying $h^{2}(\mathbb{P}^{3},I^{3}_{C}(6))=0$ form an open subset of all smooth space curves of $d=6$ and $g=3$. Hence, $h^{2}(\mathbb{P}^{3},I^{3}_{C}(6))=0$ for a generic element of the deformation class $Y$. By a similar argument, $h^{2}(\mathbb{P}^{3},I^{2}_{C}(2))\leq 8$ for a generic element of the deformation class $Y$.
\end{remark}
\begin{corollary} \label{extension of line bundles}
For a generic element of the deformation class of $Y$, we have $ext^{1}(L_{2},L_{1})=8$.
\end{corollary}

 \begin{proof}
We know that
	\[ext^{1}(L_{2},L_{1})=h^{1}(Y,L_{2}^{\vee} \otimes L_{1}),
\]
where $L_{2}^{\vee} \otimes L_{1}=\mathcal{O}_{Y}(-(3\widetilde{h})+(9\widetilde{h}-3e))=\mathcal{O}_{Y}(6\widetilde{h}-3e)$.\\
By Theorem \ref{hilbert polnomial of line bundle}, $\chi(Y,L_{2}^{\vee} \otimes L_{1})=-8$. So, we have
\[
h^{0}(L_{2}^{\vee} \otimes L_{1})-h^{1}(L_{2}^{\vee} \otimes L_{1})+h^{2}(L_{2}^{\vee} \otimes L_{1})-h^{3}(L_{2}^{\vee} \otimes L_{1})=-8
\]
\begin{eqnarray*}
	\Rightarrow h^{1}(L_{2}^{\vee} \otimes L_{1})&=&8+h^{0}(L_{2}^{\vee} \otimes L_{1})+h^{2}(L_{2}^{\vee} \otimes L_{1})-h^{3}(L_{2}^{\vee} \otimes L_{1})\\
	&=&8+\hom(L_{2},L_{1})+h^{2}(L_{2}^{\vee} \otimes L_{1})-\hom(L_{1}(1),L_{2})
\end{eqnarray*}
where $\hom(L_{2},L_{1})=\hom(L_{1}(1),L_{2})=0$ by \cite[Proposition 1.2.7]{HL10}. So
\begin{eqnarray*}	
	h^{1}(Y,L_{2}^{\vee} \otimes L_{1})&=& h^{2}(Y,L_{2}^{\vee} \otimes L_{1})+8.
\end{eqnarray*}
Use Corollary \ref{Leray spectral sequence} to compute $h^{2}(Y,L_{2}^{\vee} \otimes L_{1})$. We know that $H^{0}( \mathbb{P}^{3},R^{2} f _{\ast }L_{2}^{\vee} \otimes L_{1})=H^{0}( \mathbb{P}^{3},R^{2} f _{\ast }\mathcal{O}_{Y}(-3e)\otimes\mathcal{O}_{\mathbb{P}^{3}}(6))$ by the projection formula and $R^{2} f _{\ast } \mathcal{O}_{Y}(-3e)=0$ by Corollary \ref{direct image of exceptional divisor}.\\
 So, $H^{0}( \mathbb{P}^{3},R^{2} f _{\ast }L_{2}^{\vee} \otimes L_{1})=0$. Similarly, $H^{1}( \mathbb{P}^{3},R^{1} f _{\ast }L_{2}^{\vee} \otimes L_{1})=0$.\\
Also by the projection formula, we know that 
\begin{eqnarray*}
f_{\ast }(L_{2}^{\vee} \otimes L_{1})&=&f_{\ast }\mathcal{O}_{Y}(6 \widetilde{h}-3e)  \\
&=&\mathcal{O}_{\mathbb{P}^{3}}(6) \otimes f_{\ast }\mathcal{O}_{Y}(-3e)\\
&=&I_{C}^{3}\otimes \mathcal{O}_{\mathbb{P}^{3}}(6)\\
&=&I_{C}^{3}(6).
\end{eqnarray*}
So $H^{2}( \mathbb{P}^{3}, f _{\ast }L_{2}^{\vee} \otimes L_{1})=H^{2}( \mathbb{P}^{3},I_{C}^{3}(6))$. Hence, by Corollary \ref{Leray spectral sequence},
 	 \[H^{2}(Y,L_{2}^{\vee} \otimes L_{1})=H^{2}( \mathbb{P}^{3},I_{C}^{3}(6)).
 \]
So
\begin{eqnarray*}	
	h^{1}(Y,L_{2}^{\vee} \otimes L_{1})&=& h^{2}(\mathbb{P}^{3},I_{C}^{3}(6))+8.
\end{eqnarray*}
But, $h^{2}( \mathbb{P}^{3},I_{C}^{3}(6))=0$ by Remark \ref{open condition}, for a generic element of deformation class $Y$. Hence, $ext^{1}(L_{2},L_{1})=8$ for a generic element of deformation class $Y$.
  \end{proof}

\begin{corollary} \label{extension of line bundles 2}
For a generic element of deformation class $Y$, $ext^{1}(L_{1},L_{2})\leq8$.
\end{corollary}

 \begin{proof}
We know that
	\[ext^{1}(L_{1},L_{2})=h^{1}(Y,L_{1}^{\vee} \otimes L_{2})=h^{2}(Y,L_{2}^{\vee} \otimes L_{1} \otimes K_{Y})
\]
where $L_{2}^{\vee} \otimes L_{1} \otimes K_{Y}=\mathcal{O}_{Y}(-(3\widetilde{h})+(9\widetilde{h}-3e)+(-4\widetilde{h}+e))=\mathcal{O}_{Y}(2\widetilde{h}-2e)$. Use Corollary \ref{Leray spectral sequence} to compute $h^{2}(Y,L_{2}^{\vee} \otimes L_{1} \otimes K_{Y})$.\\
We know that $H^{0}( \mathbb{P}^{3},R^{2} f _{\ast }L_{2}^{\vee} \otimes L_{1} \otimes K_{Y})=H^{0}( \mathbb{P}^{3},R^{2} f _{\ast }\mathcal{O}_{Y}(-2e)\otimes\mathcal{O}_{\mathbb{P}^{3}}(2))$ by the projection formula and $R^{2} f _{\ast } \mathcal{O}_{Y}(-2e)=0$ by Corollary \ref{direct image of exceptional divisor}.\\
So, $H^{0}( \mathbb{P}^{3},R^{2} f _{\ast }L_{2}^{\vee} \otimes L_{1} \otimes K_{Y})=0$.\\
Similarly, $H^{1}( \mathbb{P}^{3},R^{1} f _{\ast }L_{2}^{\vee} \otimes L_{1} \otimes K_{Y})=H^{1}( \mathbb{P}^{3},R^{1} f _{\ast }\mathcal{O}_{Y}(-2e)\otimes\mathcal{O}_{\mathbb{P}^{3}}(2))$ by the projection formula and $R^{1} f _{\ast } \mathcal{O}_{Y}(-3e)=0$ by Corollary \ref{direct image of exceptional divisor}.\\
 So, $H^{1}( \mathbb{P}^{3},R^{1} f _{\ast }L_{2}^{\vee} \otimes L_{1} \otimes K_{Y})=0$.\\
Also, by the projection formula, we know that 
\begin{eqnarray*}
f_{\ast }(L_{2}^{\vee} \otimes L_{1} \otimes K_{Y})&=&f_{\ast }\mathcal{O}_{Y}(2 \widetilde{h}-2e)  \\
&=&\mathcal{O}_{\mathbb{P}^{3}}(2) \otimes f_{\ast }\mathcal{O}_{Y}(-2e)\\
&=&I_{C}^{2}\otimes \mathcal{O}_{\mathbb{P}^{3}}(2)\\
&=&I_{C}^{2}(2).
\end{eqnarray*}
So $H^{2}( \mathbb{P}^{3}, f _{\ast }L_{2}^{\vee} \otimes L_{1} \otimes K_{Y})=H^{2}( \mathbb{P}^{3},I_{C}^{2}(2))$. Hence, by Corollary \ref{Leray spectral sequence},
 	 \[H^{2}(Y,L_{2}^{\vee} \otimes L_{1} \otimes K_{Y})=H^{2}( \mathbb{P}^{3},I_{C}^{2}(2)).
 \]
But, $h^{2}( \mathbb{P}^{3},I_{C}^{2}(2))\leq 8$ by Remark \ref{open condition}, for a generic element of deformation class $Y$.\\
Hence, $ext^{1}(L_{1},L_{2})\leq 8$ for a generic element of deformation class $Y$.
  \end{proof}

\begin{theorem} \label{simple rank 2 Ulrich}
Let $\mathcal{E}$ be a rank 2 vector bundle on $Y$ obtained by a non-split extension 
	\[0 \rightarrow L_{1} \rightarrow \mathcal{E} \rightarrow L_{2} \rightarrow 0
\]
or
\[0 \rightarrow L_{2} \rightarrow \mathcal{E} \rightarrow L_{1} \rightarrow 0
\]
where $L_{1}=\mathcal{O}_{Y}(9\widetilde{h}-3e)$ and $L_{2}=\mathcal{O}_{Y}(3\widetilde{h})$. Then $\mathcal{E}$ is a simple Ulrich bundle with $c_{1}(\mathcal{E})=12\widetilde{h}-3e$ and $c_{2}(\mathcal{E})=27\widetilde{h}^{2}-9\widetilde{h}e$.
\end{theorem}

 \begin{proof}
By Theorem \ref{line bundles}, $L_{1}$ and $L_{2}$ are Ulrich line bundles. Since they are Ulrich, they have the same slope by Proposition \ref{hilbert polynoial condition}. Since they are line bundles, they are trivially stable. Clearly, they are non-isomorphic. Hence $\mathcal{E}$ is a simple vector bundle by \cite[Lemma 4.2]{CHGS12}.\\
Since $L_{1}$ and $L_{2}$ are Ulrich bundles, $\mathcal{E}$ is an Ulrich bundle by \cite[Proposition 2.8]{CKM13}.\\
Moreover, we have
\begin{eqnarray*}
c_{1}(\mathcal{E}) &=& c_{1}(L_{1})+c_{1}(L_{2})\\
&=& (9\widetilde{h}-3e) + (3\widetilde{h})\\
&=&12\widetilde{h}-3e
\end{eqnarray*}
and
\begin{eqnarray*}
c_{2}(\mathcal{E}) &=&c_{1}(L_{1})c_{1}(L_{2})\\
&=& (9\widetilde{h}-3e)(3\widetilde{h})\\
&=&27\widetilde{h}^{2}-9\widetilde{h}e.
\end{eqnarray*}
 \end{proof}

\begin{theorem} \label{quot dimension 15}
Let $\mathcal{E}$ be a rank 2 simple Ulrich bundle on $Y$ with $c_{1}(\mathcal{E})=12\widetilde{h}-3e$ and $c_{2}(\mathcal{E})=27\widetilde{h}^{2}-9\widetilde{h}e$. Then $h^{1}(\mathcal{E}\otimes\mathcal{E^{\vee}})-h^{2}(\mathcal{E}\otimes\mathcal{E^{\vee}})=15$. 
\end{theorem}

 \begin{proof}
Note that the Chern polynomial of $\mathcal{E}$ is
	\[c_{t}(\mathcal{E})=(1+(9\widetilde{h}-3e)t)(1+(3\widetilde{h})t)=\prod^{2}_{i=1}(1+a_{i}t)
\]
where $a_{1}=9\widetilde{h}-3e$ and $a_{2}=3\widetilde{h}$.\\
Also, 
\begin{eqnarray*}
c_{1}(\mathcal{E}^{\vee}) &=& (-1)^{1}c_{1}(\mathcal{E})\\
&=&-12\widetilde{h}+3e
\end{eqnarray*}
and
\begin{eqnarray*}
c_{2}(\mathcal{E}^{\vee}) &=&(-1)^{2}c_{2}(\mathcal{E})\\
&=&27\widetilde{h}^{2}-9\widetilde{h}e.
\end{eqnarray*}
Then the Chern polynomial of $\mathcal{E^{\vee}}$ is
\begin{eqnarray*}
c_{t}(\mathcal{E}^{\vee})=(1+(-9\widetilde{h}+3e)t)(1+(-3\widetilde{h})t)=\prod^{2}_{i=1}(1+b_{i}t)
\end{eqnarray*}
where $b_{1}=-(9\widetilde{h}-3e)$ and $b_{2}=-3\widetilde{h}$. Then we have 
\begin{eqnarray*}
c_{t}(\mathcal{E}\otimes\mathcal{E^{\vee}})&=&\prod_{i,j=1}^{2}(1+(a_{i}+b_{j})t)\\
&=&(1+0t)(1+(6\widetilde{h}-3e)t)(1+(-6\widetilde{h}+3e)t)(1+0t)\\
&=&1+0t+(-36\widetilde{h}^{2}+36\widetilde{h}e-9e^{2})t^{2}+0t^{3}+0t^{4}.
\end{eqnarray*}
So, $c_{2}(\mathcal{E}\otimes\mathcal{E^{\vee}})=-36\widetilde{h}^{2}+36\widetilde{h}e-9e^{2}$ and $c_{i}(\mathcal{E}\otimes\mathcal{E^{\vee}})=0$ for $i=1,3,4$.\\
By Theorem \ref{hilbert polnomial of line bundle}, we have
	\[c_{1}(\mathcal{T}_{Y})=4\widetilde{h}-e
\]
	\[c_{2}(\mathcal{T}_{Y})=12\widetilde{h}^{2}-4\widetilde{h}e
\]
and
	\[\deg(\widetilde{h}^{3})=1
\]
	\[\deg(\widetilde{h}^{2}e)=0\\
\]	
	\[\deg(\widetilde{h}e^{2})=-6
\]
	\[\deg(\widetilde{e}^{3})=-28.
\]
Apply the Riemann-Roch theorem for $\mathcal{E}\otimes\mathcal{E^{\vee}}$ on $Y$ if $c_{i}=c_{i}(\mathcal{E}\otimes\mathcal{E^{\vee}})$ and $d_{i}=c_{i}(\mathcal{T}_{\widetilde{Y}})$:
\begin{eqnarray*}
\chi(\widetilde{Y},\mathcal{E}\otimes\mathcal{E^{\vee}})&=&\frac{1}{6}(c^{3}_{1}-3c_{1}c_{2}+3c_{3})+\frac{1}{4}d_{1}(c^{2}_{1}-2c_{2})+\frac{1}{12}(d^{2}_{1}+d_{2})c_{1}+\frac{4}{24}d_{1}d_{2}\\
&=&\frac{1}{4}(4\widetilde{h}-e)(-2(-36\widetilde{h}^{2}+36\widetilde{h}e-9e^{2}))+\frac{4}{24}(4\widetilde{h}-e)(12\widetilde{h}^{2}-4\widetilde{h}e)\\
&=&\frac{1}{4}(-72)+\frac{1}{6}(24)\\
&=&-14.
\end{eqnarray*}
Then we have 
	\[
h^{0}(\mathcal{E}\otimes\mathcal{E^{\vee}})-h^{1}(\mathcal{E}\otimes\mathcal{E^{\vee}})+h^{2}(\mathcal{E}\otimes\mathcal{E^{\vee}})-h^{3}(\mathcal{E}\otimes\mathcal{E^{\vee}})=-14
\]
\begin{eqnarray*}
	\Rightarrow h^{1}(\mathcal{E}\otimes\mathcal{E^{\vee}})-h^{2}(\mathcal{E}\otimes\mathcal{E^{\vee}})&=&14+h^{0}(\mathcal{E}\otimes\mathcal{E^{\vee}})-h^{3}(\mathcal{E}\otimes\mathcal{E^{\vee}})\\
	&=&14+\hom(\mathcal{E},\mathcal{E})-\hom(\mathcal{E}(1),\mathcal{E})
\end{eqnarray*}
where $\hom(\mathcal{E},\mathcal{E})=1$ since $\mathcal{E}$ is simple. So
\begin{eqnarray*}	
	h^{1}(\mathcal{E}\otimes\mathcal{E^{\vee}})-h^{2}(\mathcal{E}\otimes\mathcal{E^{\vee}})&=&14+1-\hom(\mathcal{E}(1),\mathcal{E})
\end{eqnarray*}
where $\hom(\mathcal{E}(1),\mathcal{E})=0$ by \cite[Proposition 1.2.7]{HL10}. So
\begin{eqnarray*}	
	h^{1}(\mathcal{E}\otimes\mathcal{E^{\vee}})-h^{2}(\mathcal{E}\otimes\mathcal{E^{\vee}})&=&14+1-0\\
	&=&15.
\end{eqnarray*}
 \end{proof}

\subsection{Quot Scheme}
The general reference for this section is \cite[Section 2.2]{HL10}.\\
The Quot scheme $Quot_{X}(F,P)$ parametrizes quotient sheaves of a given $\mathcal{O}_{X}$-module $F$ with Hilbert polynomial $P$.
In this subsection, we briefly review some properties of the Quot scheme, including properties about its local dimension.

Let $\kappa$ be a field, $S$ be $\kappa$-scheme of finite type and $Sch/S$ be the category of $S$-schemes. Let $\phi:X \rightarrow S$ be a projective morphism and $\mathcal{O}_{X}(1)$ an $\phi$-ample line bundle on $X$. Let $\mathcal{H}$ be a coherent $\mathcal{O}_{X}$-module and $P\in\mathbb{Q}[z]$ a polynomial. The functor
	\[
  \mathcal{Q}:= \underline{Quot}_{X/S}:(Sch/S)^{o}\rightarrow(Sets)
\]
is defined as follows:\\
If $T \rightarrow S$ is an object in $Sch/S$,  let $\mathcal{Q}(T)$ be the set of all $T$-flat coherent quotient sheaves $\mathcal{H}_{T}=\mathcal{O}_{T}\otimes \mathcal{H} \rightarrow F$ with Hilbert poynomial $P$. And if $h:T^{'}\rightarrow T$ is an $S$-morphism, let $\mathcal{Q}(h):\mathcal{Q}(T)\rightarrow\mathcal{Q}(T^{'})$ be the map that sends $\mathcal{H}_{T}\rightarrow F$ to $\mathcal{H}_{T^{'}}\rightarrow h^{*}_{X}F$.

\begin{theorem}
The functor $\underline{Quot}_{X/S}(\mathcal{H},P)$ is represented by a projective $S$-scheme $\pi:Quot_{X/S}(\mathcal{H},P) \rightarrow S$.
\end{theorem}

 \begin{proof}
See \cite[Theorem 2.2.4]{HL10}.
 \end{proof}

\begin{proposition} \label{general bounded dimension of quot scheme}
Let $X$ be a projective scheme over a field $\kappa$ and $\mathcal{H}$ a coherent sheaf on $X$. Let $[q:\mathcal{H} \rightarrow F ] \in Quot(\mathcal{H},P)$ be a $\kappa$-rational point and $K=ker(q)$. Then 
	\[hom(K,F)\geq dim_{[q]}Quot(H,P)\geq hom(K,F)-ext^{1}(K,F).
\]
If equality holds at the second place, $Quot(\mathcal{H},P)$ is a local complete intersection near $[q]$. If $ext^{1}(K,F)=0$, then $Quot(\mathcal{H},P)$ is smooth at $[q]$.	 
\end{proposition}

 \begin{proof}
See \cite[Proposition 2.2.8]{HL10}.
 \end{proof}

\subsection{Stable Ulrich Bundles on $Y$ with $c_{1}=3H$}
We review some well-known facts.

\begin{proposition}
Let $\mathcal{E}$ be a stable bundle on $X$. Then $\mathcal{E}$ is simple; i.e, $End(\mathcal{E}) \cong \mathbb{K}$.
\end{proposition}

 \begin{proof}
 Since $\mathbb{K}$ is algebraically closed, it follows from \cite[Corollary 1.2.8]{HL10}.
 \end{proof}

\begin{theorem} \label{stability for Ulrich}
Let $\mathcal{E}$ be an Ulrich bundle of rank r on a nonsingular projective variety $X$. Then, 
\begin{itemize}
	\item $\mathcal{E}$ is semistable and $\mu$-semistable,
	\item If $\mathcal{E}$ is stable, then it is also $\mu$-stable.
\end{itemize}

\end{theorem}
 
 \begin{proof}
See \cite[Theorem 2.9]{CHGS12}.
 \end{proof}
Hence, (semi)stability and $\mu$-(semi)stability are equivalent for an Ulrich bundle $\mathcal{E}$ by Lemma \ref{stability lemma} and  Theorem \ref{stability for Ulrich}.
\begin{proposition}\label{Ulrich bundles are globally generated}
Let $\mathcal{E}$ be an Ulrich bundle of rank r on a nonsingular projective variety $X$. Then $\mathcal{E}$ is globally generated.
\end{proposition}

 \begin{proof}
See \cite[Corollary 2.5]{CKM13}.
 \end{proof}

\begin{lemma} \label{Jordan-Holder factors of Ulrich}
Let $\mathcal{E}$ be an Ulrich bundle on $X$. Then for any Jordan-H\"{o}lder filtration
	\[ 0=\mathcal{E}_{0} \subseteq \mathcal{E}_{1} \subseteq \cdots \subseteq \mathcal{E}_{m-1} \subseteq \mathcal{E}_{m}=\mathcal{E}
\]
we have that $\mathcal{E}_{i}$ is an Ulrich bundle for $1\leq i \leq m$.
In particular, if $\mathcal{E}$ is a $strictly$ $semistable$ Ulrich bundle of rank $r\geq2$, then there exist a subbundle $\mathcal{F}$ of $\mathcal{E}$ having rank $s<r$ which is Ulrich.
\end{lemma}
 \begin{proof}
See \cite[Lemma 2.15]{CKM12}.
 \end{proof}
\begin{definition}
Let $E$ be a nontrivial locally free sheaf on $X$. The trace map $tr:End(E) \rightarrow \mathcal{O}_{X}$ induces $tr^{i}:Ext^{i}(E,E)\cong H^{i}(End(E)) \rightarrow H^{i}(\mathcal{O}_{X})$. These homomorphisms are surjective. Let $Ext^{i}(E,E)_{o}$ denote the kernel of $tr^{i}$.
\end{definition}

\begin{proposition} \label{ext sub o}
If $E$ is locally free sheaf on $Y$, then $Ext^{i}(E,E)_{o}=Ext^{i}(E,E)$ for $0<i<3$.
\end{proposition}
 \begin{proof}
 Note that $H^{i}(Y,\mathcal{O}_{Y})=0$ for $0<i<3$. So the kernel of $tr^{i}$ is $Ext^{i}(E,E)$ for $0<i<3$.
 \end{proof}

We want to analyze the local dimension of Quot scheme. For this, we will follow the discussion and the notation of \cite[Section 4.3]{HL10}.\\
Let $F$ be semistable sheaf on $X$.
Let $m$ be a sufficiently large integer such that $F(m)$ is globally generated, $V$ be a vector space of dimension $P_{X}(m)$ and $\mathcal{H}:=V\otimes_{k} \mathcal{O}_{X}(-m)$. Let $R\subset Quot(\mathcal{H},P)$ be the open subscheme of those quotients $[\rho: \mathcal{H}\rightarrow \mathcal{E}]$ where $V\rightarrow H^{0}(\mathcal{E})$ is an isomorphism.
\begin{proposition} \label{Y is ACM}
$H^{i}(Y, \mathcal{O}_{Y})\cong 0$ for $i>0$.
\end{proposition}
\begin{proof}
See \cite[p. 153]{Hir64}.
\end{proof}

\begin{theorem} \label{dimension of Quot scheme at Ulrich bundle with 3H}
Let $\mathcal{E}$ be a rank 2 simple Ulrich bundle on $Y$, with $c_{1}(\mathcal{E})=12\widetilde{h}-3e$ and $c_{2}(\mathcal{E})=27\widetilde{h}^{2}-9\widetilde{h}e$. Then $\dim_{[\rho]}R \geq 1614$ for a fixed $[\rho: \mathcal{H}\rightarrow \mathcal{E}]$.
\end{theorem}

 \begin{proof}
We will follow the construction in \cite[p.115]{HL10}.\\
First, note that $\mathcal{E}$ is semistable by Theorem \ref{stability for Ulrich}. Second, $\mathcal{E}$ is globally generated by Proposition \ref{Ulrich bundles are globally generated}.\\
So $V$ is a vector space of dimension 40, since $P_{Y}(0)=20 \cdot 2\binom{3+0}{3}=40$.\\
 Then $\mathcal{H}:=V\otimes_{\mathbb{K}} \mathcal{O}_{Y}=\mathcal{O}_{Y}^{\oplus 40}$.\\
Fix $[\rho: \mathcal{H}\rightarrow \mathcal{E}]\in R$.
 \begin{enumerate}
   \item Let $K$ be the kernel of $\rho$; that is, we have
		 \[0 \rightarrow K \rightarrow \mathcal{H} \rightarrow \mathcal{E} \rightarrow 0.
	 \]
	Then we have the long exact sequence of cohomology
\begin{eqnarray*}
0 &\rightarrow& H^{0}(Y,K) \rightarrow H^{0}(Y,\mathcal{H}) \rightarrow H^{0}(Y,\mathcal{E})\\
&\rightarrow& H^{1}(Y,K) \rightarrow H^{1}(Y,\mathcal{H}) \rightarrow H^{1}(Y,\mathcal{E})\\
&\rightarrow& H^{2}(Y,K) \rightarrow H^{2}(Y,\mathcal{H}) \rightarrow H^{2}(Y,\mathcal{E})\\
&\rightarrow& H^{3}(Y,K) \rightarrow H^{3}(Y,\mathcal{H}) \rightarrow H^{3}(Y,\mathcal{E})\rightarrow 0. 
\end{eqnarray*}	
Since $\mathcal{H}=\mathcal{O}_{Y}^{\oplus 40}$ and $\mathcal{E}$ is globally generated by Proposition \ref{Ulrich bundles are globally generated},\\
$H^{0}(Y,\mathcal{H})\cong H^{0}(Y,\mathcal{E})$. So $H^{0}(Y,K)\cong 0$. Then, since $Hom(\mathcal{H},K)\cong Hom(\mathcal{O}_{Y},K)^{\oplus 40} \cong H^{0}(Y,K)^{\oplus 40}$, $Hom(\mathcal{H},K) \cong 0$.\\
Since $H^{1}(Y,\mathcal{H})\cong H^{1}(Y,\mathcal{O}_{Y})^{\oplus 40}\cong 0$ by Proposition \ref{Y is ACM}  and $H^{0}(Y,\mathcal{H})\cong H^{0}(Y,\mathcal{E})$, $H^{1}(Y,K)\cong 0$. Then, since $Ext^{1}(\mathcal{H},K)\cong Ext^{1}(\mathcal{O}_{Y},K)^{\oplus 40}\cong H^{1}(Y,K)^{\oplus 40}$, $Ext^{1}(\mathcal{H},K) \cong 0$.\\
Since $H^{2}(Y,\mathcal{H})\cong H^{2}(Y,\mathcal{O}_{Y})^{\oplus 40}\cong 0$ by Proposition \ref{Y is ACM} and $H^{1}(Y,\mathcal{E})\cong 0$ by being that $\mathcal{E}$ is Ulrich, $H^{2}(Y,K)\cong 0$. Then, since $Ext^{2}(\mathcal{H},K)\cong Ext^{2}(\mathcal{O}_{Y},K)^{\oplus 40}\cong H^{2}(Y,K)^{\oplus 40}$, $Ext^{2}(\mathcal{H},K) \cong 0$.\\
Since $H^{3}(Y,\mathcal{H})\cong H^{3}(Y,\mathcal{O}_{Y})^{\oplus 40}\cong 0$ by Proposition \ref{Y is ACM} and $H^{2}(Y,\mathcal{E})\cong 0$ by being that $\mathcal{E}$ is Ulrich, $H^{3}(Y,K)\cong 0$. Then, since $Ext^{3}(\mathcal{H},K)\cong Ext^{3}(\mathcal{O}_{Y},K)^{\oplus 40}\cong H^{3}(Y,K)^{\oplus 40}$, $Ext^{3}(\mathcal{H},K) \cong 0$.\\
Hence $Hom(\mathcal{H},K)\cong0$ and $Ext^{i}(\mathcal{H},K) \cong 0$ for $i>0$.
\item Consider the short exact sequence 
\[0 \rightarrow K \rightarrow \mathcal{H} \rightarrow \mathcal{E} \rightarrow 0.
	 \]
Then take the functor $Hom(\mathcal{H},-)$
\begin{eqnarray*}
0 &\rightarrow& Hom(\mathcal{H},K) \rightarrow Hom (\mathcal{H},\mathcal{H}) \rightarrow Hom(\mathcal{H},\mathcal{E})\\
&\rightarrow& Ext^{1}(\mathcal{H},K) \rightarrow Ext^{1}(\mathcal{H},\mathcal{H}) \rightarrow Ext^{1}(\mathcal{H},\mathcal{E})\\
&\rightarrow& Ext^{2}(\mathcal{H},K) \rightarrow Ext^{2}(\mathcal{H},\mathcal{H}) \rightarrow Ext^{2}(\mathcal{H},\mathcal{E})\\
&\rightarrow& Ext^{3}(\mathcal{H},K) \rightarrow Ext^{3}(\mathcal{H},\mathcal{H}) \rightarrow Ext^{3}(\mathcal{H},\mathcal{E})\rightarrow 0.
\end{eqnarray*}
By step (1), we know that $Hom(\mathcal{H},K)\cong0$ and $Ext^{i}(\mathcal{H},K) \cong 0$ for $i>0$. So, $Hom(\mathcal{H},\mathcal{H})\cong Hom(\mathcal{H},\mathcal{E})$ and $Ext^{i}(\mathcal{H},\mathcal{H}) \cong Ext^{i}(\mathcal{H},\mathcal{E})$ for $i>0$.\\
On the other hand, $Ext^{i}(\mathcal{H},\mathcal{H})\cong Ext^{i}(\mathcal{O}^{\oplus 40}_{Y},\mathcal{O}^{\oplus 40}_{Y}) \cong H^{i}(Y,\mathcal{O}_{Y})^{\oplus 1600}$ for $i>0$. Since $H^{i}(Y,\mathcal{O}_{Y})\cong 0$ for $i>0$ by Proposition \ref{Y is ACM}, $Ext^{i}(\mathcal{H},\mathcal{H})\cong 0$.\\
Hence $Hom(\mathcal{H},\mathcal{H})\cong Hom(\mathcal{H},\mathcal{E})$ and $Ext^{i}(\mathcal{H},\mathcal{E})=0$, $i>0$.
 \item Again consider the short exact sequence 
		 \[0 \rightarrow K \rightarrow \mathcal{H} \rightarrow \mathcal{E} \rightarrow 0.
	 \]
Then take the functor $Hom(-,\mathcal{E})$ of it
\begin{eqnarray*}
0 &\rightarrow& Hom(\mathcal{E},\mathcal{E}) \rightarrow Hom (\mathcal{H},\mathcal{E}) \rightarrow Hom(K,\mathcal{E})\\
&\rightarrow& Ext^{1}(\mathcal{E},\mathcal{E}) \rightarrow Ext^{1}(\mathcal{H},\mathcal{E})=0 \rightarrow \ldots
\end{eqnarray*}
leads to equality $hom(K,\mathcal{E})=hom(\mathcal{H},\mathcal{E})+ext^{1}(\mathcal{E},\mathcal{E})-hom(\mathcal{E},\mathcal{E})$.\\
Since $Ext^{i}(\mathcal{H},\mathcal{E})=0$ for $i>0$, $Ext^{i}(K,\mathcal{E})\cong Ext^{i+1}(\mathcal{E},\mathcal{E})$ for $i>0$.
  \item The boundary map $Ext^{1}(K,\mathcal{E})\rightarrow Ext^{2}(\mathcal{E},\mathcal{E})$ maps the obstruction to extend $[\rho]$ onto the obstructions to extend $[\mathcal{E}]$ (see \cite[2.A.8]{HL10}). The latter is contained in the subspace $Ext^{2}(\mathcal{E},\mathcal{E})_{o}$. This gives the dimension bound, using Proposition \ref{general bounded dimension of quot scheme},
		\[\dim_{[\rho]}R \geq hom(K,\mathcal{E})-ext^{2}(\mathcal{E},\mathcal{E})_{o}.
	\]
	Then, by step (3), we have
		\[\dim_{[\rho]}R \geq hom(\mathcal{H},\mathcal{E})+ext^{1}(\mathcal{E},\mathcal{E})-hom(\mathcal{E},\mathcal{E})-ext^{2}(\mathcal{E},\mathcal{E})_{o}.
	\]
	Then, by Proposition \ref{ext sub o}
		\[\dim_{[\rho]}R \geq hom(\mathcal{H},\mathcal{E})+ext^{1}(\mathcal{E},\mathcal{E})-hom(\mathcal{E},\mathcal{E})-ext^{2}(\mathcal{E},\mathcal{E}).
	\]
	Then, by step (2), we have
		\[\dim_{[\rho]}R \geq hom(\mathcal{H},\mathcal{H})+ext^{1}(\mathcal{E},\mathcal{E})-hom(\mathcal{E},\mathcal{E})-ext^{2}(\mathcal{E},\mathcal{E}).
	\]
	Since $\mathcal{E}$ is simple and $\mathcal{H}=\mathcal{O}_{Y}^{\oplus 40}$, we have
		\[\dim_{[\rho]}R \geq 1600+ext^{1}(\mathcal{E},\mathcal{E})-1-ext^{2}(\mathcal{E},\mathcal{E}).
	\]
	Then, by Theorem \ref{quot dimension 15} and the equality $h^{i}(\mathcal{E} \otimes \mathcal{E}^{\vee})=ext^{i}(\mathcal{E},\mathcal{E})$, we have
		\[\dim_{[\rho]}R \geq 1600+-1+15 = 1614.
	\]
\end{enumerate}
\end{proof}
Let $R^{'}\subset Quot(\mathcal{H},P)$ be the subset parametrizing the quotients $[\rho: \mathcal{H}\rightarrow \mathcal{E}]$ where $\mathcal{E}$ is obtained as an extension of $L_{2}$ by $L_{1}$.
\begin{proposition} \label{dimension of quot scheme by extensions}
$\dim_{[\rho]}R^{'} = 1606$ for a fixed $[\rho: \mathcal{H}\rightarrow \mathcal{E}]$.
\end{proposition}
 \begin{proof}
 The projectivization of $Ext^{1}(L_{2},L_{1})$ has dimension $8-1=7$ by Corollary \ref{extension of line bundles}. $R^{'}$ is the union of all orbits of extensions of $L_{2}$ by $L_{1}$ under the action of $PGL(V)$, so around each fixed $[\rho: \mathcal{H}\rightarrow \mathcal{E}]$,  $\dim_{[\rho]}R^{'} = 1599+7=1606$.
 \end{proof}
Let $R^{''}\subset Quot(\mathcal{H},P)$ be the subset parametrizing the quotients $[\rho: \mathcal{H}\rightarrow \mathcal{E}]$ where $\mathcal{E}$ is obtained as an extension of $L_{1}$ by $L_{2}$.
\begin{proposition} \label{dimension of quot scheme by extensions 2}
$\dim_{[\rho]}R^{''} \leq 1606$ for a fixed $[\rho: \mathcal{H}\rightarrow \mathcal{E}]$.
\end{proposition}
 \begin{proof}
 The projectivization of $Ext^{1}(L_{1},L_{2})$ has dimension $\leq8-1=7$ by Corollary \ref{extension of line bundles 2}. $R^{''}$ is the union of all orbits of extensions of $L_{1}$ by $L_{2}$ under the action of $PGL(V)$, so around each fixed $[\rho: \mathcal{H}\rightarrow \mathcal{E}]$,  $\dim_{[\rho]}R^{''} \leq 1599+7=1606$.
 \end{proof}
\begin{theorem} \label{stable rank 2 Ulrich bundle}
There exist rank 2 stable Ulrich bundles with $c_{1}(\mathcal{E})=12\widetilde{h}-3e$ on a generic element of the deformation class $Y$. 
\end{theorem}

 \begin{proof}
By Theorem \ref{simple rank 2 Ulrich}, there are rank 2 simple Ulrich bundle $\mathcal{E}$ with the given Chern classes.\\
We know that the property of being Ulrich is an open condition. So there is an open subset $U$ of $R$ around $[\rho: \mathcal{H}\rightarrow \mathcal{E}]$ containing Ulrich bundles. By Theorem \ref{dimension of Quot scheme at Ulrich bundle with 3H}, $U$ has dimension at least 1614.\\
We also know that every Ulrich bundle is semistable by Theorem \ref{stability for Ulrich}. If all elements of $U$ were strictly semistable, then by Lemma \ref{Jordan-Holder factors of Ulrich} and \cite[Proposition 2.8]{CKM13}, they would be extensions of Ulrich line bundles. But there are only two Ulrich line bundles $L_{1}$ and $L_{2}$ on $Y$. So they would be extensions of $L_{2}$ by $L_{1}$ or extensions of $L_{1}$ by $L_{2}$.\\
However, the dimension of $R^{'}$ at the points that are extensions of $L_{2}$ by $L_{1}$ is 1606 by Proposition \ref{dimension of quot scheme by extensions} and the dimension of $R^{''}$ at the points that are extensions of $L_{1}$ by $L_{2}$ is at most 1606 by Proposition \ref{dimension of quot scheme by extensions 2}. Since both these dimensions are strictly smaller than 1614, not all Ulrich bundles with the given Chern classes are obtained by extensions. In other words, not all Ulrich bundles in $U$ are strictly semistable. Hence there are rank 2 stable Ulrich bundles with $c_{1}(\mathcal{E})=12\widetilde{h}-3e$. 
 \end{proof}

\end{document}